\documentclass[preprint]{elsarticle}

\usepackage{amsmath,amssymb}
\usepackage{amsthm}
\usepackage{latexsym}
\usepackage{graphicx}
\usepackage{color}
\usepackage{subfigure}
\usepackage{hyperref}

\newcommand{\weps}{\widetilde{\epsilon}}

\newcommand{\NN}{\mathbb{N}}
\newcommand{\ZZ}{\mathbb{Z}}

\newcommand{\RR}{\mathbb{R}}
\newcommand{\CC}{\mathbb{C}}

\newcommand{\ud}{\textup{d}}

\newcommand{\SSS}{\mathcal{S}}

\newcommand{\TT}{\mathbb{T}}
\newcommand{\ECG}{\textup{ECG}}
\newcommand{\II}{\textup{II}}
\newcommand{\I}{\textup{I}}
\newcommand{\IHR}{\textup{IHR}}
\newcommand{\IRR}{\textup{IRR}}
\newcommand{\Resp}{\textup{Resp}}

\theoremstyle{definition}
\newtheorem{theorem}{Theorem}[section]
\newtheorem{defn}[theorem]{Definition}

\newtheorem{lemma}[theorem]{Lemma}

\newtheorem{cor}{Corollary}[theorem]

\theoremstyle{remark}
\newtheorem*{remark}{Remark}

\begin{document}
\begin{frontmatter}
\title{Instantaneous frequency and wave shape functions (I)}

\author{Hau-tieng~Wu}
\ead{hauwu@math.princeton.edu}

\address{Department of Mathematics, Princeton University, Fine Hall, Washington Road, Princeton NJ 08544-1000 USA.}

\begin{abstract}
Although one can formulate an intuitive notion of instantaneous frequency, generalizing "frequency" as we understand it in e.g. the Fourier transform, a rigorous mathematical definition is lacking. In this paper, we consider a class of functions composed of waveforms that repeat nearly periodically, and for
which the instantaneous frequency can be given a rigorous meaning. We show that Synchrosqueezing can be used to determine the instantaneous frequency of functions in this class, even if the waveform is not harmonic, thus generalizing earlier results for cosine wave functions. We also provide real-life examples and discuss the advantages, for these examples, of considering such non-harmonic waveforms.
\end{abstract}

\end{frontmatter}

\section{Introduction}
The term ``instantaneous frequency'' is somewhat of an oxymoron. In many cases, researchers who seek to decompose signals into different components, and who want to determine the ``instantaneous frequency''\cite{picinbono:1997,picinbono_martin:1983} of each, really seek to solve problems of the following form:

{\em given a function of the form
\begin{equation}\label{foft}
f(t)=\sum_{k=1}^K A_k(t)\cos(2\pi\phi_k(t)),\mbox{ with } A_k(t),\phi'_k(t)>0 ~ \forall t,
\end{equation}
compute $A_k(t)$ and $\phi'_k(t)$ and/or describe their properties.}

The Synchrosqueezing transform, a reassignment method \cite{kodera_gendrin_villedary:1978,auger_flandrin:1995} originally introduced in the context of audio signal analysis \cite{Daubechies1996}, and further analyzed in \cite{Daubechies2010,Thakur2010,wu_flandrin_daubechies:2011,brevdo_fuckar_thakur_wu:2012}, provides a way to determine $A_k(t)$ and $\phi'_k(t)$ uniquely, up to some pre-assigned accuracy, under some conditions on $f(t)$.

For some applications, decompositions of the form (\ref{foft}) are too restrictive. Consider, for instance, the function $f$ illustrated in the left plot of Figure \ref{shape0}(a); it is a toy example of the type
\begin{equation}\label{toy1}
f(t)=A(t)s(2\pi\phi(t))
\end{equation}
where, as before, $A(t)$ and $\phi'(t)$ vary slowly, and $s$ is now no longer a cosine, but
the
periodic extension of the function in the right plot of Figure \ref{shape0}(a).
It is clear that this $f(t)$ can be decomposed as in (\ref{foft}), simply by replacing $s$ by its Fourier expansion; however, the representation (\ref{toy1}) is much more efficient because it uses fewer terms.
Another way in which one could reduce $f(t)$ to an expansion of type (\ref{foft}), with a {\em single} term, would be to ``absorb'' some of the properties of $s$ into a modified phase function
$m(t)\,:=\,\cos^{-1}\left(s(2\pi\phi(t)) \right)$.
This is the preferred solution in \cite{HuangWuLong:09}, where $m(t)$ is then called the ``intra-wave'' modulation, indicative of the nonlinear nature of the wave process generating $f(t)$.
Although this would work for the example in Figure \ref{shape0}(a), it is not always possible to do this.
Consider the function $f$ in the left plot of Figure \ref{shape0}(b), clearly of the same type as in Figure \ref{shape0}(a); one would expect that the same remarks apply to both examples.
However, because $s$ has several maxima, it can not be written as $\cos( m(t))$, with $m$ a monotonic mapping on $[0,2\pi]$. A single-term expansion of type (\ref{foft}) can thus not deal with this example by using only ``intra-wave modulation''.
To nevertheless still reduce $f(t)$ to an expansion of type (\ref{foft}) with a single term,
one can, instead, attempt to ``absorb'' some of the properties of $s$ into a modified amplitude $\widetilde{A}$.  In this particular case,
\begin{equation}\label{shapeL1}
s(t)=\left[\cos(0.8\cos(t))-\frac{\sin(t)}{\cos(t)}\sin(0.8\cos(t))-1.4\cos_+(t+1/6)^2\right]\cos(t)=:w(t)\cos(t),
\end{equation}
where $\cos_+(t)=\max\{0,\cos(t)\}$. So one can write $f(t)=\widetilde{A}(t)\cos(2\pi\phi(t))$, with $\widetilde{A}(t)=w(2\pi\phi(t))A(t)$. Note that this amplitude $\widetilde{A}$ varies much faster than $A$, masking the only slowly changing wave-pattern of $f(t)$.
Even this solution is not always applicable.
Consider the function $f$ plotted on the left of Figure \ref{shape0}(c). Absorbing the ``extraneous'' extrema of the basic wave-pattern into a special $\tilde{A}$ would lead to an amplitude that is no longer always positive, clearly an undesirable trait.

Although the examples in Figure \ref{shape0} are just toy examples, similar phenomena can be observed in real-life signals, for example, in electrocardiography (ECG) (see Figure \ref{figex1}).
As we shall see below, it is important, in these real-life examples, to tease apart the characteristics of the ``shape'' $s(t)$ from the slow variations in $\phi'(t)$ and $A(t)$. Rather than insisting on a representation of type (\ref{foft}), we  are therefore, in this paper, interested in decompositions of the type
\begin{equation}\label{foft2}
f(t)=\sum_{k=1}^K A_k(t)s_k(2\pi\phi_k(t)),
\end{equation}
where we shall give the name ``wave-shape function'' (or, shorter, ``shape function'') to the $2\pi$ periodic functions $s_k$, generalizing the cosine functions of (\ref{foft}) and \cite{Daubechies2010}, and where we assume that $|A_k'(t)|$ and $|\phi''_k(t)|$ are small compared with $\phi'_k(t)$, as in \cite{Daubechies2010}. Functions of type (\ref{foft2}) can be found in many applications. We already mentioned ECG signals; another medical signal of this same form is respiration.

The additional layer of generality in (\ref{foft2}), when compared to (\ref{foft}), adds to the complexity of determining desirable decompositions of type (\ref{foft2}) for a (noisy) signal
$f$. Even for decompositions of type (\ref{foft}), uniqueness is not guaranteed
(see e.g. \cite{Daubechies2010}); this absence of unqueness can obviously
be only more pronounced for decompositions of type (\ref{foft2}).
Even when the decomposition is unambiguous, the algorithmic task will necessarily be more complex, since not only the $A_k(t)$ and $\phi_k(t)$, but also the $s_k(t)$ need to be determined, in general.
We shall concentrate here on the determination of the amplitudes $A_k(t)$, the instantaneous frequencies $\phi'_k(t)$ and the components $A_k(t)s_k(2\pi\phi_k(t))$ for the more generalized expansions (\ref{foft2})  for a restricted class of wave shape functions. For more general wave shape functions $s_k$, the determination of the wave shapes $s_k$ themselves will be discussed in a sub-sequential paper.

The main result of this paper is that, under the same technical conditions as in \cite{Daubechies2010} (the $\phi_k'$ have to be sufficiently separated), we can determine $A_k(t)$ and $\phi'_k(t)$ and reconstruct each component via Synchrosqueezing, for suitable wave forms $s_k$.  We shall apply this method to ECG and respiration signals, and show the clinical potential of the results. In the next section, we first discuss some properties of these signals in more detail. Section 3 states and proves the related theorems, which generalize the results provided in \cite{Daubechies2010}.

\begin{figure}[h]\label{shape0}
\begin{minipage}{.4 in}
(a)
\vspace*{2.5 cm}

(b)
\vspace*{2.5 cm}

(c)
\end{minipage}
\begin{minipage}{15 cm}
\subfigure{\resizebox{14cm}{9cm}{\includegraphics{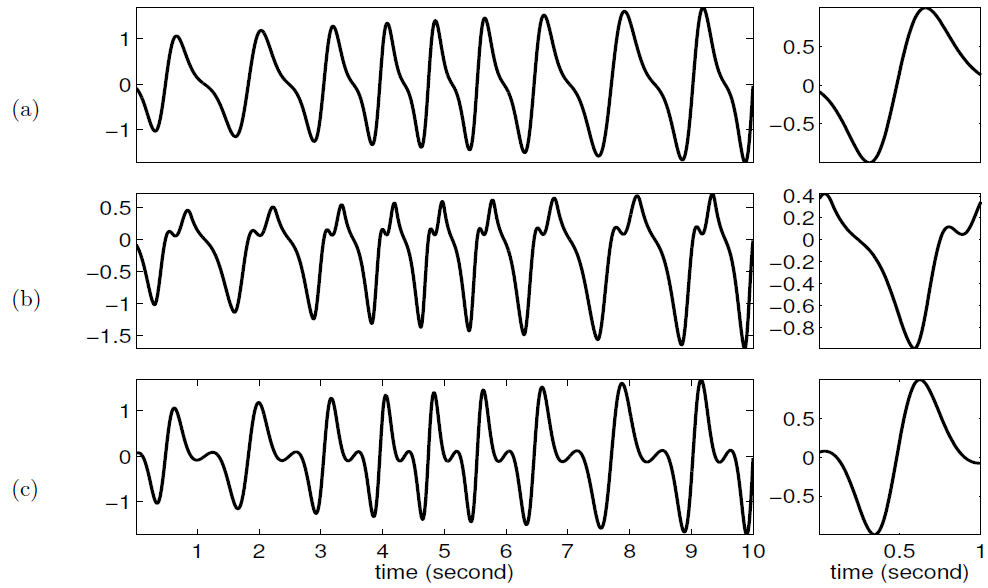}}}
\end{minipage}
\caption{Toy examples of the type $f(t)=A(t)s(2\pi \phi(t))$. The amplitude modulation function and the phase function are identical for the three cases:
$A(t)=\sqrt{1+0.2t}$ and $\phi(t)=t+0.3\cos(t)$. In (a) the shape function $s(t)$ is given
by $s(t)=\left[\cos(0.8\cos(t))-\frac{\sin(t)}{\cos(t)}\sin(0.8\cos(t))\right]\cos(t)$; in (b): $s(t)=\left[\cos(0.8\cos(t))-\frac{\sin(t)}{\cos(t)}\sin(0.8\cos(t))-1.4\cos_+(t+1/6)^2\right]\cos(t)$; in (c): $s(t)=\left[\cos(1.2\cos(t))-\frac{\sin(t)}{\cos(t)}\sin(1.2\cos(t))\right]\cos(t)$.}
\end{figure}

\section{Two biomedical signals}

The ECG signal can be measured easily and cheaply, and its clinical usefulness is well established.
It is commonly accepted to model the ECG signal as a current dipole vector undergoing a periodic
motion in $\RR^3 $ \cite{keener1998}; the recorded ECG signal is then viewed as
the orthogonal projection of this
dipole vector onto a fixed axis.\footnote{In fact, the situation is a bit more complex. In addition
to the periodic motion for every heartbeat, the trajectory in $\RR^3$ of the current dipole vector
is deformed by other phenomena such as breathing. We
shall ignore these effects here.} The ECG signal, recording the dynamics of the electrical activity
of the heart, is a collection of periodic oscillating time series, one per channel, each corresponding to an ECG lead \cite{Guyton:00,ecg}. The {\em waveforms} provide a lot of
information about the anatomic or electrophysiological structure of the heart, essential in
ascertaining certain medical conditions, such as ischemia or atrial fibrillation. On the other hand, the
{\em variation of the time intervals between sequential heart beats}, referred to as ``heart rate
variability'' (HRV), has been shown in the past few decades to be related to more general physiological
dynamical processes \cite{hrv}. To understand these different types of dynamics of the physiological
system via the recorded ECG signal, it thus is beneficial to separate the shape of the oscillation
from the variability of the time intervals between sequential oscillatory waveforms.

In this section, the example ECG signals are recorded from a healthy 33-year-old male, with 12 bit resolution and a sampling rate of 1000Hz. We show in Figure \ref{figex1} the lead I and lead II ECG signals, which we denote as $\ECG_\I(t)$ and $\ECG_{\II}(t)$, $t\in [0,T]$.
\begin{figure}[h]
\subfigure{
\includegraphics[width=0.95\textwidth]{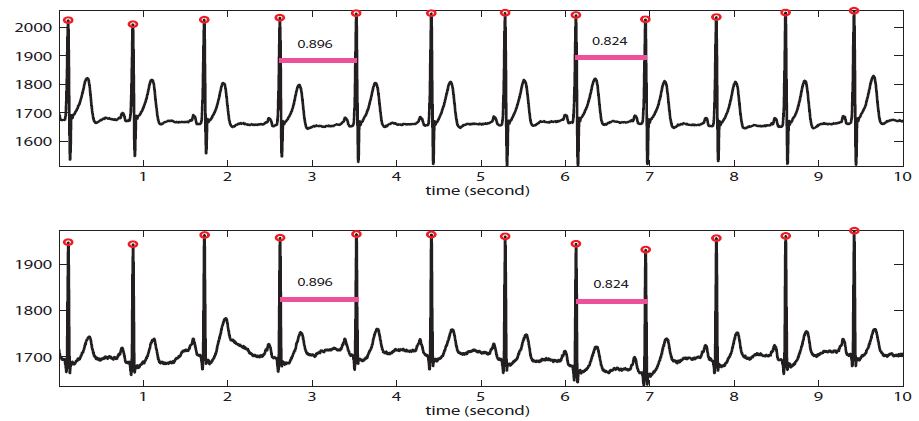}}
\caption{Top: the lead I ECG signal; bottom: the lead II ECG signal, where the red circles mark the locations of the R peaks. It is clear that the time interval between consecutive R-peaks is not constant.}\label{figex1}
\end{figure}
In the figure, the black curve is the usual ECG signal; the peaks indicated by red circles are
called the R peaks \cite{ecg}. One type of variation from one cycle to the next in the ECG signal
can be clearly tracked by the changing time intervals between consecutive R-peaks. In clinical
practice, the ``heart rate'' is given by simply counting the number of beats during a minute, that
is, the ``mean rate'' over a given time period. However, there is information hidden inside the HRV,
beyond the mean rate, that we want to describe quantitatively. To access this, let us first introduce
an intuitive definition of time-dependent instantaneous heart rate (IHR), as the inverse of the time
interval between the two most recent successive heart beats. We refer to this quantity as the {\em
intuitive instantaneous heart rate} (or {\em intuitive instantaneous frequency}) and denote it as $\IHR_i(t)$, where the subscript $i$ refers to the ``intuitive'' character of this definition.
Denoting by $t_k$ the location of the R-peak of the $k$-th heart beat, $k=1,\ldots,N$. $\IHR_i$ is
defined as the piecewise constant function plotted in Figure \ref{figex2},
\[
\IHR_i(t)=\frac{1}{t_{k}-t_{k-1}}\mbox{ if }t_k\leq t<t_{k+1}.
\]
Note that the definition of $\IHR_i$ echoes the etymology of  ``frequency'': counting how {\em frequently} a phenomenon occurs per unit time.
\begin{figure}[h]
\subfigure{
\includegraphics[width=0.95\textwidth]{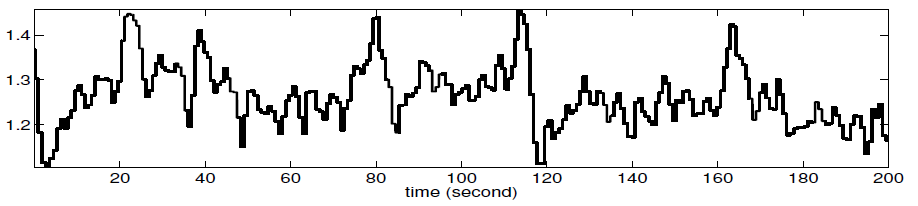}}
\caption{The intuitive instantaneous heart rate.}\label{figex2}
\end{figure}

To a very good approximation, the ECG signals can be modeled as
\[
\ECG_\ell(t)=A_\ell(t)s_\ell(\phi_E(t))+W_\ell(t),
\]
where $\ell$ stands for the lead (or channel) number (e.g., $\I$ or $\II$), and where $\phi_E(t)$ does not depend on $\ell$. (The subscript $E$ indicates that this is the $\phi(t)$ for the ECG signal; we shall meet other $\phi$ below.) The function $W_\ell(t)$ contains all the other components of the signal, including noise and low frequency baseline wandering.
We show in this paper that the Synchrosqueezing transform \footnote{The official code of the Synchrosqueezing transform is available in \url{http://www.math.princeton.edu/~ebrevdo/synsq/}} can identify $\phi_E'(t)$ for such ECG signals, with high accuracy, independently of the detailed properties of the wave shape functions $s_\ell(t)$, for $s_\ell$ within a certain class of functions (that contains ECG profiles).
Physiologically, the HRV reflects the periodic behavior of the current dipole vector; it should
thus not depend on how we project this dipole vector. In other words, our estimate should be
the same, whether we derive it from $\ECG_\I(t)$ or $\ECG_{\II}(t)$, i.e. whether the behavior in
time of the wave shape function it modulates is associated with lead I or lead II. It turns out that
the Synchrosqueezing estimate does indeed exhibit this invariance, as we demonstrate numerically
here. (A mathematical explanation of this will be given in Section \ref{theorem}.)  We denote the
estimated $\phi'_E$ as $\IHR_{\ell}(t)$ if it is estimated by computing the Synchrosqueezing
transform of $\ECG_{\ell}(t)$. Figures \ref{ecgIest}, resp. \ref{ecgIIest}  plot $\IHR_{\I}(t)$,
resp. $\IHR_{\II}(t)$ and show that in both cases $\IHR_{\ell}(t)$ is extremely close to $\IHR_i$.
In Figure \ref{ecgihrcompare},  $\IHR_{\I}$ and $\IHR_{\II}$ are plotted together, illustrating
directly that our estimate of $\phi'_E$ does not depend on which wave shape function was being
modulated.
(That the two shape functions $s_\I$ and $s_\II$ are different is clear from Fig.\ref{figex1}.)
Furthermore, this example indicates that our estimate
captures the  ``instantaneous frequency'': indeed, $\phi'_E$ matches $\IHR_i$
with high accuracy. This finding makes it possible to define an ``instantaneous heart rate'' for subjects with much less clearly defined heart beats: even if the $\IHR_i(t)$ can no longer be defined because the $t_k$ cannot be defined clearly, we can still estimate $\IHR_{\ell}(t)$.


\begin{figure}[h]
\subfigure{
\includegraphics[width=0.95\textwidth]{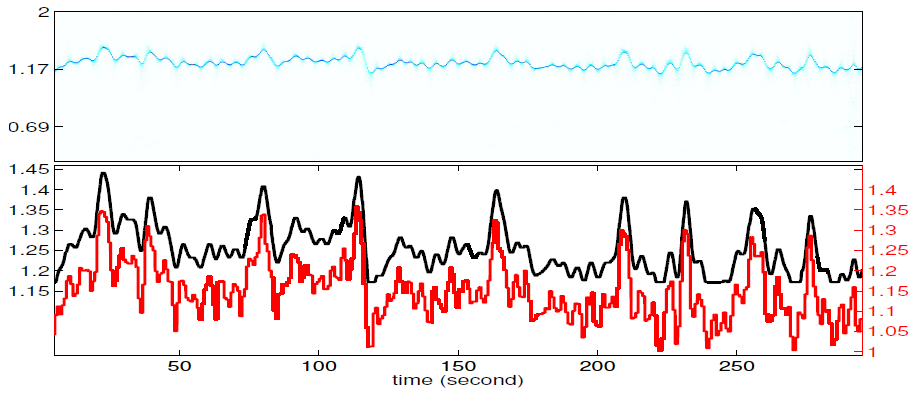}}
    \caption{Top: the Synchrosqueezing transform of $\ECG_{\I}(t)$; bottom: the black curve is $\IHR_{\I}(t)$ estimated from the Synchrosqueezing transform of $\ECG_{\I}(t)$ and the red curve is $\IHR_i(t)$ shifted down by $0.1$. }\label{ecgIest}
\end{figure}

\begin{figure}[h]
\subfigure{
\includegraphics[width=0.95\textwidth]{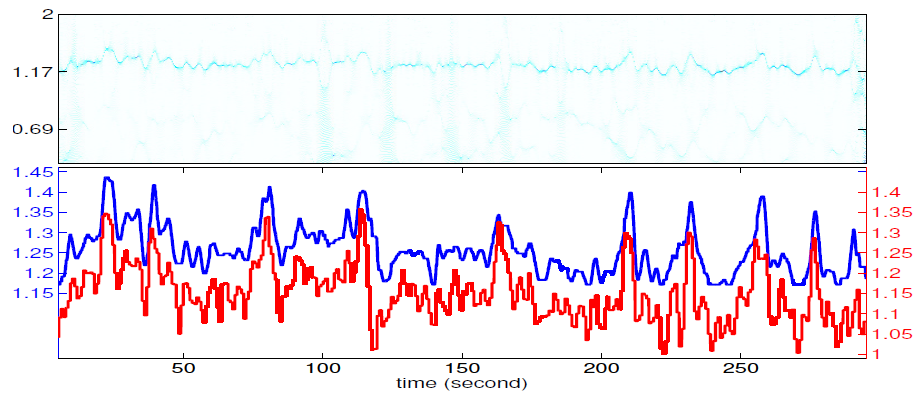}}
    \caption{Top: the Synchrosqueezing transform of $\ECG_{\II}(t)$; bottom: the blue curve is $\IHR_{II}(t)$ estimated from the Synchrosqueezing transform of $\ECG_{\II}(t)$ and the red curve is $\IHR_i(t)$ shifted down by $0.1$.}\label{ecgIIest}
\end{figure}

\begin{figure}[h]
\subfigure{
\includegraphics[width=0.95\textwidth]{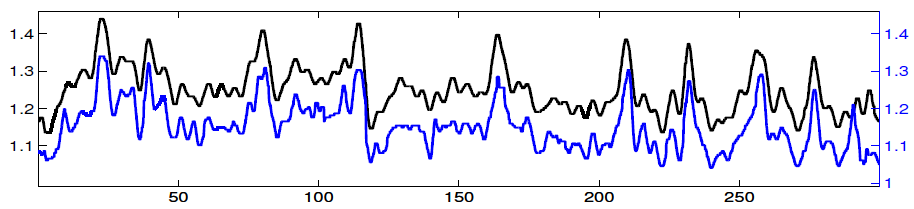}}
    \caption{The blue curve is $\IHR_{I}(t)$ and the black curve is $\IHR_{\II}(t)$ shifted down by $0.1$.}\label{ecgihrcompare}
\end{figure}

Our second example concerns the respiratory signal, i.e., the signal provided by a mechanical recording of the breathing process. Physiologically, the morphology of the recorded respiratory signal reflects the anatomical structure of our respiratory system \cite{Guyton:00}. In a first approximation, breathing is a periodic process; a closer look reveals that consecutive ``breathing peaks'' are not quite equally spaced (see Figure \ref{irri} below). This breathing rate variability (BRV) reflects the physiological dynamics \cite{Benchetrit2000,Wysocki2006,kou:2011}. An example of a respiratory signal $\Resp(t)$ is shown in Figure \ref{irri}. The signal was recorded from a 28-year-old healthy male, with a sampling rate of 20 Hz. Like the ECG signal, to a very good approximation, the respiratory signal can be modeled as
\[
\Resp(t)=A(t)s(\phi_R(t))+W(t),
\]
where the subscript $R$ indicates that this is the phase function $\phi(t)$ for the {\em respiratory} signal; the function $W(t)$ contains again all the other components of the signal, including noise and low frequency baseline wandering.

As in the discussion of the ECG signal, we first define an intuitive notion of time-dependent instantaneous breathing rate $\IRR_i(t)$ as the inverse of the time interval between the two most recent successive breathing cycles. A breathing cycle is defined to be the signal between two consecutive ends of the inspiration; these ends are marked as the red circles in Figure \ref{irri}. More precisely, we denote the location of the ends of the inspiration by $t_k$, $k=1,\ldots,M$ and define $\IRR_i(t)$ as follows:
$$
\IRR_i(t) = \frac{1}{t_k-t_{k-1}}\mbox{ when }t_k\leq t<t_{k+1}.
$$
The $\IRR_i(t)$ is plotted as the black piecewise constant curve in Figure \ref{irri}. Next, we apply the Synchrosqueezing transform directly on $\Resp(t)$ and get an estimation of its instantaneous frequency, denoted by $\IRR(t)$. In Figure \ref{irri}, the $\IRR(t)$ is superimposed on $\Resp(t)$ and $\IRR_i(t)$ to demonstrate that the estimated $\IRR(t)$ captures the notion of instantaneous frequency. Indeed, the spacing of respiration cycles in $\Resp(t)$ is reflected by $\IRR(t)$: closer spacing corresponds to higher $\IRR(t)$ values, and wider spacing to lower $\IRR(t)$ values. In Figure \ref{irrcompare} $\IRR(t)$ and $\IRR_i(t)$ are put together for comparison; they are clearly closely related.

\begin{figure}[h]
\subfigure{
\includegraphics[width=0.95\textwidth]{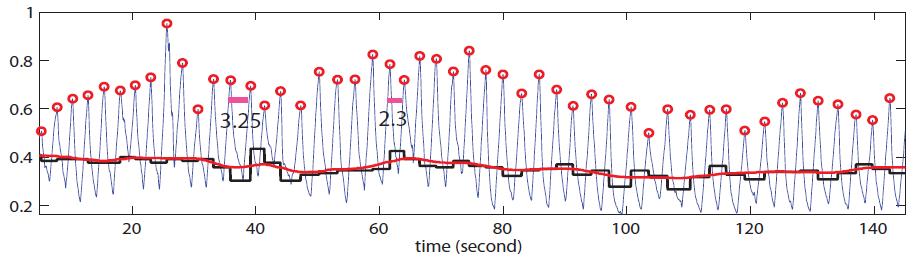}}
\caption{The blue curve is the respiratory signal $\Resp(t)$; the black piecewise constant curve is the intuitive respiration rate $\IRR_i(t)$; the smooth red curve is the $\IRR(t)$ estimated via Synchrosqueezing (or SSTIF). The vertical axis is for $\IRR(t)$ and $\IRR_i(t)$.}\label{irri}
\end{figure}

\begin{figure}[h]
\subfigure{
\includegraphics[width=0.95\textwidth]{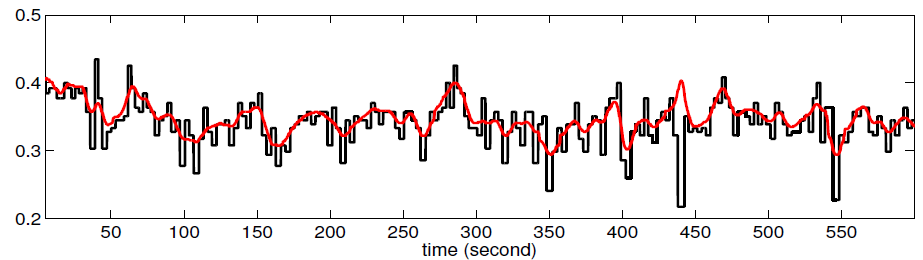}}
\caption{The red curve is the estimated instantaneous frequency $\IRR(t)$ of the respiratory signal from the Synchrosqueezing transform, and the black piecewise constant curve is the intuitive respiration rate $\IRR_i(t)$.}\label{irrcompare}
\end{figure}

For the sake of convenience, we shall use the acronym SSTIF for the {\em \textbf{S}ynchro\textbf{S}queezing \textbf{T}ransform-derived \textbf{I}nstantaneous \textbf{F}requency} in what follows. Figure \ref{sleepirr} shows a different breathing signal, and illustrates a direct physiological application of the respiration SSTIF. Sleep is a universally recurring physiological dynamical process. It is divided into two broad stages: rapid eye movement (REM) and non-rapid eye movement (NREM). Normally, sleep proceeds in cycles, each alternating between REM and NREM, with one cycle taking about 90 minutes. A clinically acceptable staging of the sleep is determined by reading the recorded electroencephalography (EEG) based on the R\&K criteria, which were standardized in 1968 by Allan Rechtschaffen and Anthony Kales \cite{RK}. We take the staging according to this criteria as the gold standard, and we demonstrate that the respiration SSTIF can recover the gold standard staging of REM versus non-REM. The result is illustrated in Figure \ref{sleepirr} which shows the
SSTIF of a respiratory signal recorded for about 8 hours with a 16 Hz sampling rate (original signal not shown here), and the corresponding sleep stages, denoted by $S(t)$, determined from a simultaneously recorded EEG according to the R\&K criteria. High correlation is observed when comparing the time intervals with $S(t)=5$ with SSTIF. \footnote{Further study and finding of the sleep cycle is beyond our scope here; a more detailed study will be presented in a later paper.}

The discussion and examples in this Section illustrate the importance of an accurate determination of the ``instantaneous frequency'' of a signal, as captured by our SSTIF notion. In principle, we can determine instantaneous frequencies $\phi'(t)$ from representation of type (1) as well as from (\ref{toy1}) or (\ref{foft2}). We believe (and will argue in the next Section) that, at least for certain signals $f(t)$, modeling $f$ as in (\ref{toy1}) or (\ref{foft2}), i.e., in terms of ``wave shape functions'' rather than with cosines, leads to more accurate estimates of the function $\phi'(t)$ via Synchrosqueezing, and probably also by other methods. For such signals, it is thus important to separate each component of $f$ modeled as in (\ref{toy1}) or (\ref{foft2}) into a ``wave shape function'' on the one hand, and slowly varying amplitude and instantaneous frequency on the other hand.

\begin{figure}[h]
\begin{centering}
\subfigure{
\includegraphics[width=0.95\textwidth]{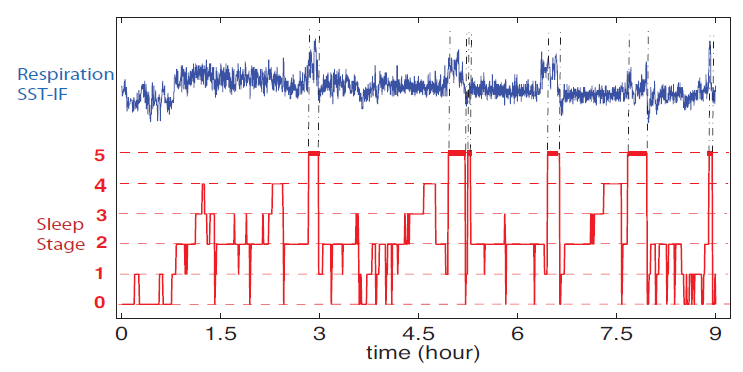}}
\end{centering}
\caption{The blue curve is the $10\times$SSTIF+4, and the red piecewise function is the sleep stage determined from the simultaneously recorded EEG by the R\&K criteria. The REM stage period is emphasized in both signals. }
\label{sleepirr}
\end{figure}

Another reason why it is important to separate ``shape function'' from ``instantaneous frequency'' is that the information hidden in the shape function is important in its own right, and can, for medical signals, be used for clinically quite different diagnoses. In other words, this separation allows us to tease apart two types of information that are commonly mixed-up. For example, reading ECG signals to diagnose cardiac disease in clinical practice amounts to evaluating the morphology of each heart beat, and this is given by the wave shape function. A well known example is the typical ``ST elevation'' in myocardial infarction patients. Another example is the qualitatively different spectral behavior of the morphology of the ECG shape function indicative of myocardial ischemia: the spectral analysis of ECG wave shape functions from dogs revealed a shift from high- to low-frequency ranges in ischemia cases \cite{mor1990}. Similar phenomena have been associated with balloon inflations during percutaneous transcatheter angioplasty in CAD patients \cite{abboud1987}. We expect that a cleaner identification of the SSTIF will also lead to a separation of the shape function particularities with greater sensitivity, which in turn will be useful for clinical diagnoses linked to those particularities. A detailed discussion of this will have to wait for a subsequent paper; here we concentrate on the identification of the instantaneous frequency.

\section{Theorem}\label{theorem}
We start by introducing some notations and conventions. Suppose $f$ is a tempered distribution. The Fourier transforms of $f$ is denoted by $\hat{f}$ and satisfies the normalization $\widehat{e^{-\pi x^{2}}}=e^{-\pi\xi^{2}}$. In the following, we denote $W_f(a,b)$ to be the continuous wavelet transform of $f$ with the mother wavelet $\psi$ which is in the Schwartz space \cite{ID1992}. For simplicity,  we shall assume $\mbox{supp}~\widehat{\psi}\subset[1-\Delta,1+\Delta]$, where $0<\Delta\ll1$; in practice, our results seem to hold under much less stringent conditions.

As is customary, we say that $\tau$ is a
period for the function
$\varphi: \RR \rightarrow \CC$, or that $\varphi$ is $\tau-$periodic, if, for all $t \in \RR$, and all $k \in \ZZ$, $\varphi(t+k\tau)=\varphi(t)$.
We shall designate $T$ to be {\em the} Period for $\varphi$,
if
\[
\forall t \in \RR\, ,
\forall k \in \ZZ\,:\,\varphi(t+kT)=\varphi(t)\, \mbox{ and }
\,
T=\inf\left(\{\, \tau\,;\, \tau >0 \mbox{ and }\tau \mbox{ is a period for } \varphi \,\} \right)\,.
\]

Next, we define a special class of wave shape functions that are dominated by one particular Fourier mode. More precisely,
\begin{defn}[Analytic shape function class $\SSS^{\delta,D,\theta}$]
Fix $\delta\geq 0$, $D\in \NN$ and $\theta\geq 0$. The analytic shape function class $\SSS^{\delta,D,\theta}$ is the subset of $C^{1,\alpha}(\TT)$, where $\TT$ is the $1$-dim torus and $\alpha>1/2$ (the $2\pi-$periodic $\alpha$-Holder continuously differentiable functions), consisting of analytic functions $s$
with mean $0$, i.e. $\hat{s}(k)=0$ for all $k\leq 0$, and unit $L^2-$norm for which all the Fourier modes $\widehat{s}(k)$, $k\neq 1$ are dominated by the product of $\delta$ and the first mode coefficient, i.e.
\begin{equation}\label{definition:shape:1}
\forall k \in \NN, \, \mbox{ with }k\neq 1, \, \left| \widehat{s}(k) \right| \leq \delta \, \left| \widehat{s}(1) \right|
\end{equation}
and 
\begin{equation}\label{definition:shape:2}
\sum_{n>D}|n\widehat{s}(n)|\leq \theta
\end{equation}
\end{defn}

\begin{remark} 
(1) The parameters $\delta$, $D$ and $\theta$ are used to characterize the ``shape'' of the shape function. Condition (\ref{definition:shape:1}) says that the ``base'' frequency $\hat{s}(1)$ can not be zero; we use this condition to estimate the ``instantaneous frequency'' as we should see below. Condition (\ref{definition:shape:2}) says that ``essentially'' the shape does not oscillate too fast. (2) In practice, we shall consider real-valued shape function $\tilde{s}$, which are not analytic. However, we can use the standard trick of viewing them as the real part of an analytic function, i.e., $\tilde{s}=\Re s$, where $\hat{s}(k)=\hat{\tilde{s}}(k)$ if $k\geq 0$, $\hat{s}(k)=0$ if $k<0$. Because the Synchrosqueezing transform (to be defined soon) uses the positive frequency domain only, $\tilde{s}$ and $s$ can be used interchangeably in our analysis. Thus, we can define a \textit{real shape function class} $\tilde{\SSS}^{\delta,D,\theta}$ of parameter $\delta$, $D$, and $\theta$ to contain the functions $\tilde{s}:\TT\to\RR$ so that $\tilde{s}=\Re s$ for some $s\in \mathcal{\SSS}^{\delta,D,\theta}$. Since the analysis of the functions with the real shape functions is the same as that of the analytic shape functions, in the following we focus on analytic shape functions in our analysis, for notational convenience, but when we demonstrate numerical results, we use real shape functions, which is easier for visualization. 
\end{remark}

For example, $e^{i t}$ is a shape function, which is widely used in Fourier analysis. Indeed, $\widehat{e^{i t}}(n)=1$ when $n=1$ and $0$ when $n\neq 1$ so that $e^{i t}\in \SSS^{0,1,0}$. We demonstrate some examples in Figure \ref{shape3}: the shape function $s_{\II}(t)$ for the ECG lead II signal satisfies $s_{\II}(t)\in \SSS^{\delta, D, \theta}$ where $\delta\approx 3.4$, $D\approx 40$ and $\theta\ll1$. (Note that the shape function $s_{\II}(t)$ depends on the subject; different ECG signals may belong to different $\tilde{\SSS}^{\delta,D,\theta}$; the value $\delta=3.4$ is one of the larger we have encountered -- more often $\delta$ can be picked less than $1$.)


\begin{figure}[h]
\subfigure{
\includegraphics[width=0.95\textwidth]{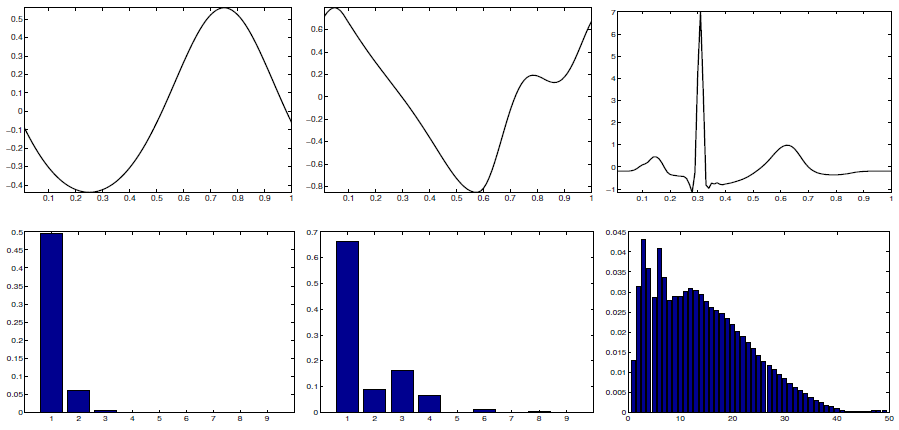}}
\caption{Three different shape functions $s_1(t)$, $s_2(t)$ and $s_{\II}(t)$ (upper) and their spectrum (below). Left:  $s_1(t)$ is equal to normalizing $\left[\cos(0.5\cos(t))-\frac{\sin(t)}{\cos(t)}\sin(0.5\cos(t))\right]\cos(t)$ to be of unit $L^2$ norm; middle: $s_2(t)$ is equal to normalizing (\ref{shapeL1}) to be of unit $L^2$ norm; right: $s_{\II}$ is a shape function for a particular lead II ECG signal.}\label{shape3}
\end{figure}
Next, consider the following class of functions.
\begin{defn}[Intrinsic Mode Functions class $\mathcal{C}^{\delta,D}_{\epsilon}$]
For fixed choices of  $\epsilon,\,\delta,\, D>0$, $\epsilon\ll 1$, the space $\mathcal{C}^{\delta,D}_{\epsilon}$ of \textit{Intrinsic Mode
Functions (IMFs)} consists of functions $f:\RR\to\CC$, $f\in C^1(\RR)\cap L^\infty(\RR)$ having the
form \begin{equation}
f(t)=A(t)s(2\pi \phi(t)),
\end{equation}
where $s\in \SSS^{\delta,D,\epsilon}$, such that $A$ and $\phi$ satisfy the following conditions:
\[A\in C^1(\RR)\cap L^\infty(\RR),~ \phi\in C^2(\RR),\]
\[
\inf_{t\in\RR}A(t)>0,~\inf_{t\in\RR}\phi'(t)>0, ~\sup_{t\in\RR}\phi'(t)<\infty\]
\[
|A'(t)|\leq\epsilon|\phi'(t)|,\quad |\phi''(t)|\leq\epsilon|\phi'(t)|\mbox{ for all }t\in\RR,\]
\[
M''\,:=\,\sup_{t \in \RR}|\phi''(t)|< \infty.
\]
\end{defn}

We then consider the function class $\mathcal{C}^{\delta,D}_{\epsilon,d}$, defined as follows.
\begin{defn}[Superpositions of IMFs]\label{DefBClass}
 The space $\mathcal{C}^{\delta,D}_{\epsilon,d}$ of \textit{superpositions
of IMFs} consists of functions $f$ having the form
\[
f(t)=\sum_{k=1}^{K}f_{k}(t)
\]
for some finite $K>0$ and $f_{k}(t)=A_{k}(t)s_k(2\pi \phi_{k}(t))\in\mathcal{C}^{\delta,D}_{\epsilon}$ such that $\phi_{k}$ satisfy 
\[
\phi'_k(t)>\phi'_{k-1}(t)\mbox{    and    }\phi'_{k}(t)-\phi'_{k-1}(t)\geq d [\phi'_k(t)+\phi'_{k-1}(t)].
\]
\end{defn}

Note that $\mathcal{C}^{\delta,D}_{\epsilon,d}$ and $\SSS^{\delta,D,\theta}$ are not vector spaces and are not the most general possible definition. We use these definitions for the purpose of simplifying the proof and clarifying the main idea of this work. 

In the following we prove that the Synchrosqueezing transform discussed in \cite{Daubechies2010} allows us to estimate the instantaneous frequency of each component of the functions in $\mathcal{C}^{\delta,D}_{\epsilon,d}$ with high accuracy when all the wave shape functions are close to the exponential function, i.e. when $\delta$ is sufficiently small. Moreover, under some constraints on the IF of each component, we are able to reconstruct each component. 
Most of the estimates needed to prove these results follow from Section 3 in \cite{Daubechies2010}; we shall spell out the details only when care has to be taken with extra terms involving the possibly more complex shape of the wave shape functions.
\begin{defn}[Instantaneous frequency information function]
Let $f\in L^\infty(\RR)$. The {\em instantaneous frequency information function} of $f$ is defined by
\[
\omega_{f}(a,b)=
\begin{cases} \frac{-i\partial_{b}W_{f}(a,b)}{2\pi W_{f}(a,b)} & |W_{f}(a,b)|>0\\
\infty & |W_{f}(a,b)|=0
\end{cases}.
\]
\end{defn}
Note that this definition makes sense since $W_f(a,b)\in C^\infty(\RR_+\times\RR)$.
\begin{remark}
In practice, the determination of those $(a,b)$-pairs for which $W_f(a,b)=0$ is rather unstable when $f$ has been contaminated by noise. For this reason, it is useful to consider a threshold for $|W_f(a,b)|$, below which $\omega_f(a,b)$ is not defined. The purpose of this function is to record the information of the instantaneous frequency, based on which the reassignment will be performed.
\end{remark}

Then, we consider the following definition of the Synchrosqueezing based on the wavelet transform:
\begin{defn}[Synchrosqueezing transform]
For $f\in L^\infty(\RR)$, the \textit{Synchrosqueezing transform with resolution $\alpha>0$ and threshold $\gamma\geq0$ } is defined by
\begin{equation}\label{ss}
S^{\alpha,\gamma}_f(b,\xi):=\int_{A_{\gamma,f}(b)} W_f(a,b) \frac{1}{\alpha}h\left(\frac{|\xi -\omega_f(a,b)|}{\alpha}\right)a^{-3/2}\ud a,
\end{equation}
where $b\in\RR$, $\xi\in\RR^+$, $A_{\gamma,f}(b):=\{a \in \RR_+;|W_f(a,b)| >\gamma\}$ and $h(t) =\frac{1}{\sqrt{\pi}}e^{-t^2}$.
\end{defn}


Note that $\omega_f(a,b)$ and $S^{\alpha,\gamma}_f(b,\xi)$ can be defined for any $L^\infty$ function $f$. However, when $f\in \mathcal{C}^{\delta,D}_{\epsilon,d}$, the next Theorem tells us that the Synchrosqueezing transform provides an accurate estimation of the instantaneous frequency and allows the reconstruction of each component.

\begin{theorem}\label{mainthm1}
Let $f(t)=\sum_{k=1}^KA_{k}(t)s_k(2\pi \phi_{k}(t))\in \mathcal{C}^{\delta,D}_{\epsilon,d}$. Suppose $\Delta<d/(1+d)$ and
\begin{equation}\label{thm:special1}
\left(\cup_{n=1}^D\left[\frac{1-d}{n\phi'_k(t)},\frac{1+d}{n\phi_k'(t)}\right]\right)\cap\left(\cup_{l<k}\cup_{n=1}^D\left[\frac{1-d}{n\phi'_l(t)},\frac{1+d}{n\phi_l'(t)}\right]\right)=\emptyset
\end{equation}
for all $1\leq k\leq K$ and $t\in\RR$. Set $\mathcal{R}_{\psi}:=\int\overline{\widehat{\psi}(z)}z^{-1}\ud z$, $Z_{k,n}:=\{(a,b):|an\phi'_k(b)-1|<\Delta\}$ and $\weps:=\epsilon^{1/3}$. 
Then, provided $\epsilon$ is sufficiently small, the following hold:
\begin{itemize}
\item $|W_f(a,b)|>\weps$ only when $(a,b) \in Z_{k,n}$ for some $k\in\{1,\ldots,K\}$ and $n\in\{1,\ldots,D\}$.
\item For each pair $(a,b) \in Z_{k,n}$ for which $|W_f(a,b)|>\weps$, we have
    \[
    |\omega_f(a,b)- n\phi_k'(b)|\leq\weps.
    \]
\item Moreover, for each $k \in \{1,\ldots,K\}$ and all $b \in \RR$,  
\[
\left|\lim_{\alpha\rightarrow 0}\left(\mathcal{R}_{\psi}^{-1}\int_{\cup_{n=1}^D\{\xi:\,|\xi-n\phi'_k(b)|<\weps\}}S^{\alpha,\weps}_{f}(b,\xi)\ud\xi\right) -A_k(b)s_k(2\pi \phi_k(b))\right|\leq C\weps,
\]
where $C=A_k(b)\weps^2+4\left[\left(\frac{\phi'_k(b)}{1-\Delta}\right)^{1/2}-\left(\frac{\phi'_k(b)}{1+\Delta}\right)^{1/2}\right]$. 
\end{itemize}
\end{theorem}

Notice that the special case when $K=1$ is important since the ECG signal and the respiration signal are both signals of this kind. The results of estimating the instantaneous frequency via Synchrosqueezing transform have been shown in Section 2.

When the wave shape functions of all the components of $f(t)$ are close to the imaginary exponential function, the following corollary immediately follows by taking $D=1$ in Theorem \ref{mainthm1}.
\begin{cor}\label{maincor}
Let $f(t)=\sum_{k=1}^KA_{k}(t)s_k(2\pi \phi_{k}(t))\in \mathcal{C}^{\delta,1}_{\epsilon,d}$. Suppose $\Delta<d/(1+d)$, and $\delta\leq\epsilon$. We use the same notations $\mathcal{R}_{\psi}$, $Z_{k,n}$ and $\weps$ as in Theorem \ref{mainthm1}. Then, provided $\epsilon$ is sufficiently small, the following hold:
\begin{itemize}
\item $|W_f(a,b)|>\weps$ only when $(a,b) \in Z_{k,1}$ for some $k\in\{1,\ldots,K\}$.
\item For each pair $(a,b) \in Z_{k,1}$ for which $|W_f(a,b)|>\weps$, we have
    \[
    |\omega_f(a,b)- \phi_k'(b)|\leq\weps.
    \]
\item For each $k \in \{1,\ldots,K\}$ and all $b \in \RR$, 
\[
\left|\lim_{\alpha\rightarrow 0}\left(\mathcal{R}_{\psi}^{-1}\int_{\{\xi:\,|\xi-\phi'_k(b)|<\weps\}}S^{\alpha,\weps}_{f}(b,\xi)\ud\xi\right) - A_k(b)s_k(2\pi\phi_k(b))\right|\leq C\weps,
\]
where $C=A_k(b)\weps^2+4\left[\left(\frac{\phi'_k(b)}{1-\Delta}\right)^{1/2}-\left(\frac{\phi'_k(b)}{1+\Delta}\right)^{1/2}\right]$. 
\end{itemize}
\end{cor}

Notice that in Corollary \ref{maincor}, the wave shape functions of all components are close to $e^{it}$ while the separation of the ``base frequencies'', that is, $\phi_k'(t)$, are not further limited; on the other hand, the wave shape functions of each component in Theorem \ref{mainthm1} is more general than the cosine function but the behaviors of $\phi_k'(t)$ need to be further limited by (\ref{thm:special1}). Setting $s_k(t)=e^{it}$ for all $k=1,\ldots,K$, we fully recover Theorem 3.3 from \cite{Daubechies2010}.



Condition (\ref{thm:special1}) can be slightly relaxed as follows. Suppose  $f(t)=\sum_{k=1}^KA_{k}(t)s_k(2\pi \phi_{k}(t))\in \mathcal{C}^{\delta,D}_{\epsilon,d}$, where $s_k\in \SSS^{\delta,D_k,\epsilon}$ and $D_k\leq D$ maybe different. Then the following condition suffices for the proof:
\begin{equation}\label{thm:special2}
\left(\cup_{n=1}^{D_k}\left[\frac{1-d}{n\phi'_k(t)},\frac{1+d}{n\phi_k'(t)}\right]\right)\cap\left(\cup_{l<k}\cup_{n=1}^{D_l}\left[\frac{1-d}{n\phi'_l(t)},\frac{1+d}{n\phi_l'(t)}\right]\right)=\emptyset.
\end{equation}
Since the proof is entirely analogous to what we show below, involving longer computations without making a conceptual difference, we focus here on the condition (\ref{thm:special1}) only.

The proof of the Theorem is similar to the one carried out in \cite[Theorem 3.3]{Daubechies2010}, except for some estimates related to the wave shape function, which are given in detail below. In the statement and proof of all the Lemmas, we shall always assume that the conditions of Theorem \ref{mainthm1} are satisfied without repeating them, unless stated otherwise.

First of all, we have the following estimates bounding the growth of $A_k(t)$ and $\phi'_k(t)$ in terms of the value of $\phi_k'(t)$. The proof is the same as that of Estimate 3.4 in \cite{Daubechies2010}, so we omit it.
\begin{lemma}\label{est00}
For each $k \in \{1,\ldots,K\}$, we have
\[
|A_k(t+u)\,-\,A_k(t)| \,\leq \,
\epsilon\,|u|\,\left(\,|\phi'_k(t)|\,+\,\frac{1}{2}\,M''_k\,|u|
  \,\right)
\]
\[ ~~ \mbox{ and } ~~ |\phi'_k(t+u)\,-\,\phi'_k(t)| \,\leq
\,\epsilon\,|u|\,\left(\,|\phi'_k(t)|\,+\,\frac{1}{2}\,M''_k\,|u|
  \,\right) ~.
\]
\label{est0}
\end{lemma}

The next lemma concerns the ``dyadic separation'' condition in the definition of $\mathcal{C}^{\delta,D}_{\epsilon,d}$. It implies that for every $(a,b)$-pair, at most one component of the signal ``comes into play''. 

\begin{lemma}\label{est0}
For any pair $(a,b)$ such that $a>\frac{1-\Delta}{D\phi_K'(b)}$, there can be at most one $l \in \{1,\ldots,K\}$ and $n\in\{1,\ldots,D\}$ for which $|an\phi'_l(b)-1| <\Delta$.
\end{lemma}

\begin{proof}
Suppose (\ref{thm:special1}) holds. We rewrite $f(t)$ as
\begin{equation}\label{est0:eq0}
\sum_{k=1}^KA_{k}(t)s_k(2\pi \phi_{k}(t)) = \sum_{k=1}^K\sum_{n=1}^D A_{k}(t)\hat{s}_k(n)e^{i2\pi n\phi_{k}(t)} + \sum_{k=1}^K\sum_{n=D+1}^\infty A_{k}(t)\hat{s}_k(n)e^{i2\pi n\phi_{k}(t)},
\end{equation}
where the pointwise convergence holds since $s_k\in C^{1,\alpha}$. By condition (\ref{thm:special1}), we know for all $t\in\RR$
\begin{equation}\label{est0:eq1}
\left[\frac{1-d}{n\phi'_l(t)},\frac{1+d}{n\phi'_l(t)}\right]\cap\left[\frac{1-d}{m\phi'_{j}(t)},\frac{1+d}{m\phi'_{j}(t)}\right]=\emptyset
\end{equation}
when $l\neq j$ or $n\neq m$, where $n,m\leq D$.  
Next we permute $D\times K$ components $\{n\phi'_k(b)\}_{n=1,\ldots,D,\, k=1,\ldots,K}$ in the ascending way and index $\phi'_1(b)$ by $1$ and $D\phi'_K(b)$ by $D\times K$. Thus (\ref{est0:eq0}) can be further rewritten as
\[
\sum_{l=1}^{D\times K} \tilde{A}_{l}(t)e^{i2\pi \psi_{l}(t)} + \sum_{k=1}^K\sum_{n=D+1}^\infty A_{k}(t)\hat{s}_k(n)e^{i2\pi n\phi_{k}(t)},
\]
where $\tilde{A}_1(t)=\hat{s}_1(1)A_1(t),\ldots,\tilde{A}_{D\times K}(t)=\hat{s}_K(D)A_K(t)$ and $\psi_1(t)=\phi_1(t),\ldots,\psi_{D\times K}(t)=D\phi_K(t)$. Take $j,l\in\{1,\ldots,D\times K\}$ so that $l>j$ and fix $a>\frac{1-\Delta}{D\phi_K'(b)}=\frac{1-\Delta}{\psi_{D\times K}'(b)}$. By (\ref{est0:eq1}) we have
\begin{equation}\label{est0:eq2}
\psi'_l(b)-\psi'_{j}(b)  \geq d [\psi'_j(b)+\psi'_{l}(b)] 
\end{equation}
since $\frac{1+d}{\psi'_{l}(b)}<\frac{1-d}{\psi'_{j}(b)}$. 
Now suppose 
$| a \psi'_j(b) - 1 |  < \Delta$ and $| a \psi'_{l}(b) - 1 |  < \Delta$ hold simultaneously. A direct calculation leads to 
\[
  \psi'_l(b)-\psi'_{j}(b)  \leq  \frac{ (1+\Delta) - (1-\Delta) }{a}  = \frac{2 \Delta}{a}
\]
and
\[
  \psi'_l(b)+\psi'_{j}(b)  \geq  \frac{ (1-\Delta) + (1-\Delta) }{a}  = \frac{2 (1-\Delta)}{a},
\]
which with (\ref{est0:eq2}) gives
\[
 \Delta  \geq  d (1-\Delta)~,
\]
contradicting the condition $\Delta < \frac{d}{1+d}$.
\end{proof}

{\allowdisplaybreaks 

The following Lemma describes the result of applying the continuous wavelet transform to $f\in \mathcal{C}^{\delta,D}_{\epsilon,d}$. It is similar to Estimate 3.5 in \cite{Daubechies2010} except for some extra terms due to the more general form of the wave shape functions; when $s_k(t)=e^{i t}$ for all $k\in\{1,\ldots,K\}$, we recover the statement and proof of Estimate 3.5 in \cite{Daubechies2010}.
\begin{lemma}\label{est1}
For $l\in\{1,\ldots,K\}$, $n\in\{1,\ldots,D\}$ and $(a,b)\in Z_{l,n}$, we have
\begin{equation}\label{est1:rslt1}
\left|W_f(a,b)-A_l(b)\hat{s}_l(n)e^{i2\pi n\phi_l(b)}\sqrt{a}\overline{\widehat{\psi}\left(an\phi'_l(b)\right)}\right|\leq
\epsilon a^{1/2}\Lambda_1(a,b).
\end{equation}
where
\begin{equation}
\Lambda_1(a,b)=\sum_{k=1}^K \Big\{\|s_k\|_\infty \left(\phi_k'(b)aI_1+\frac{M''_k}{2} a^{2}I_2\right)+\pi A_k(b) \left[\sum_{j\in\ZZ} |j|\left|\hat{s}_k(j)\right|\right] \left(a^{2}I_2|\phi'_k(b)|+\frac{M_k''}{3}a^{3}I_3\right)\Big\}+ I_0\sum_{k=1}^{l-1} A_k(b) \nonumber
\end{equation}
and $I_i=\int_\RR |x|^i|\psi(x)|\ud x$.
\end{lemma}

\begin{remark}  
We observe from this Lemma that when (\ref{thm:special1}) holds, the larger the $l$ is, the larger the $\Lambda_1(a,b)$ in Lemma \ref{est1}. Indeed, take $f=\sum_{k=1}^Kf_k\in\mathcal{C}^{\delta,D}_{\epsilon,d}$ and fix $f_l$ for some $l\in\{1,\ldots,K\}$ for example. The larger $l$ is, the more $\phi'_k(b)$ will be smaller than $\phi'_l(b)$ so that more $\hat{\psi}(a\phi_k'(b)n)$, $n>D$ become (possibly) positive. Consider $a$ such that $a\phi'_l(b)\in[1-\Delta,1+\Delta]$. For $k\in\{1,\ldots,K\}$, $k<l$, $\hat{\psi}(a\phi_k'(b)n)\neq 0$ holds only if $\frac{n\phi'_k(b)(1-\Delta)}{\phi'_l(b)}<1+\Delta$ or $1-\Delta<\frac{n\phi'_k(b)(1+\Delta)}{\phi'_l(b)}$, or equivalently
\begin{equation}\label{Nkl}
\frac{1-\Delta}{1+\Delta}\frac{\phi'_l(b)}{\phi'_k(b)}<n<\frac{1+\Delta}{1-\Delta}\frac{\phi'_l(b)}{\phi'_k(b)}.
\end{equation}
Clearly when the $l$ is large, $\hat{\psi}(a\phi_k'(b)n)\neq 0$ for more combinations of $k$ and $n$.
On the other hand, if $s_k(t)\neq e^{i t}$ for some $k<l$, due to the existence of the nonzero high Fourier modes of the shape function $s_k(t)$, that is, $\hat{s}_k(n)\neq 0$ for $n>1$, the nonzero $\hat{\psi}(a\phi_k'(b)n)$ will survive for those $n$ for which $\hat{s}_k(n)\neq 0$. Thus, the continuous wavelet coefficients around $Z_{l,n}$, $n=1,\ldots,D$, will be ``contaminated''. Putting these two effects together, we expect that the larger $l$ is and the more non-zero $\hat{s}_k(n)$, $n>D$ and $k<l$ are, the larger the $\Lambda_1(a,b)$ will be. Thus, the higher the $n>1$ is, the more blurring is around the band $Z_{k,n}$ on the time-frequency plane, which depend on the wave shape functions. It is for these reasons that we shall evaluate the IF from the {\em lowest} mode $\widehat{s}_k(1)$ of the $s_k$. We shall illustrate this below, in examples.
\end{remark}

\begin{proof}
Since $f\in \mathcal{C}^{\delta,D}_{\epsilon,d}\subset L^\infty(\RR)$ and the mother wavelet $\psi(t)$ is a Schwartz function, $W_f(a,b)$ is well defined. We work out the estimation step by step. First, we evaluate the following integration for $k\in\{1,\ldots,K\}$ and $(a,b)\in \RR_+\times\RR$:
\begin{eqnarray}
& &\int_\RR A_k(b)s_k\Big(2\pi(\phi_k(b)-b\phi_k'(b))+2\pi\phi_k'(b)t\Big)\frac{1}{\sqrt{a}}\overline{\psi\left(\frac{t-b}{a}\right)}\ud t\label{Clemma1:eq1}\\
&=& A_k(b)\int_\RR s_k(2\pi(\phi_k(b)-b\phi_k'(b))+2\pi x)\frac{1}{\sqrt{a}}\overline{\psi\left(\frac{x-b\phi_k'(b)}{a\phi_k'(b)}\right)}\frac{1}{\phi_k'(b)}\ud x\nonumber\\
&=& A_k(b)\int_\RR s_k(2\pi(\phi_k(b)-u))\frac{1}{\sqrt{a}}\overline{\psi\left(\frac{u}{a\phi_k'(b)}\right)}\frac{1}{\phi_k'(b)}\ud u\nonumber\\
&=& A_k(b)\sum_{n\in\ZZ}\hat{s}_k(n)\int_\RR e^{i2\pi n(\phi_k(b)-u)}\frac{1}{\sqrt{a}}\overline{\psi\left(\frac{u}{a\phi_k'(b)}\right)}\frac{1}{\phi_k'(b)}\ud u\nonumber\\
&=& \sum_{n\in\NN}A_k(b)\hat{s}_k(n) e^{i2\pi n\phi_k(b)}\sqrt{a}\overline{\hat{\psi}(a\phi_k'(b)n)}\nonumber,
\end{eqnarray}
where the first two equalities comes from the change of variable and the third equality comes from the fact that $s_k\in C^{1,\alpha}(\mathbb{T})$.

Next, by applying Lemma \ref{est00}, we can evaluate the difference between the continuous wavelet transform of $A_k(b)s_k(2\pi\phi_k(t))$ and $A_k(b)s_k\Big(2\pi(\phi_k(b)-b\phi_k'(b))+2\pi\phi_k'(b)t\Big)$:
\begin{eqnarray}
&&\left|\int_\RR A_k(b)s_k(2\pi\phi_k(t))\frac{1}{\sqrt{a}}\overline{\psi\left(\frac{t-b}{a}\right)}\ud t\right.\nonumber\\
&&\left.\qquad -\int_\RR A_k(b)s_k\Big(2\pi(\phi_k(b)-b\phi_k'(b))+2\pi\phi_k'(b)t\Big)\frac{1}{\sqrt{a}}\overline{\psi\left(\frac{t-b}{a}\right)}\ud t\right|\nonumber\\
&=&\left|\int_\RR A_k(b)s_k\left[ 2\pi\left(\phi_k(b)-b\phi_k'(b)+\phi_k'(b)t+\int^{t-b}_0\left[\phi_k'(b+u)-\phi_k'(b)\right]\ud u\right)\right]\frac{1}{\sqrt{a}}\overline{\psi\left(\frac{t-b}{a}\right)}\ud t\right.\nonumber\\
&&\left.\qquad -\int_\RR A_k(b)s_k\Big[2\pi\big(\phi_k(b)-b\phi_k'(b)+\phi_k'(b)t\big)\Big]\frac{1}{\sqrt{a}}\overline{\psi\left(\frac{t-b}{a}\right)}\ud t\right|\nonumber\\
&=&A_k(b)\left|\int_\RR \left[\sum_{n\in\ZZ} \hat{s}_k(n)e^{i2\pi n\Big(\phi_k(b)-b\phi_k'(b)+\phi_k'(b)t\Big)}\left(e^{i2\pi n \int^{t-b}_0\left[\phi_k'(b+u)-\phi_k'(b)\right]\ud u}-1\right)\right]\frac{1}{\sqrt{a}}\overline{\psi\left(\frac{t-b}{a}\right)}\ud t\right|\nonumber\\
&\leq&A_k(b)\int_\RR \left[\sum_{n\in\ZZ} \left|\hat{s}_k(n)\right|\left|e^{i2\pi n \int^{t-b}_0\left[\phi_k'(b+u)-\phi_k'(b)\right]\ud u}-1\right|\right]\frac{1}{\sqrt{a}}\left|\overline{\psi\left(\frac{t-b}{a}\right)}\right|\ud t\nonumber\\
&\leq&2\pi A_k(b)\int_\RR \left[\sum_{n\in\ZZ} |n|\left|\hat{s}_k(n)\right|\left| \int^{t-b}_0\left[\phi_k'(b+u)-\phi_k'(b)\right]\ud u\right|\right]\frac{1}{\sqrt{a}}\left|\overline{\psi\left(\frac{t-b}{a}\right)}\right|\ud t\nonumber\\
&\leq&2\pi \epsilon A_k(b) \left[\sum_{n\in\ZZ} |n|\left|\hat{s}_k(n)\right|\right] \int_\RR \left[\frac{1}{2}|t-b|^2|\phi'_k(b)|+\frac{1}{6}|t-b|^3M_k''\right]\frac{1}{\sqrt{a}}\left|\overline{\psi\left(\frac{t-b}{a}\right)}\right|\ud t\nonumber\\
&\leq&2\pi \epsilon A_k(b) \left[\sum_{n\in\ZZ} |n|\left|\hat{s}_k(n)\right|\right] \left(\frac{1}{2}a^{5/2}I_2|\phi'_k(b)|+\frac{1}{6}M_k''a^{7/2}I_3\right)\nonumber
\end{eqnarray}
where the first equality comes from applying Taylor's expansion to $\phi_k(t)$ and the fifth inequality comes from Lemma \ref{est00}. Note that $\sum_{n\in\ZZ} |n|\left|\hat{s}_k(n)\right|<\infty$ since $s_k\in C^{1,\alpha}$ with $\alpha>1/2$.

Third, we approximate $A_k(t)$ by Lemma \ref{est00}:
\begin{eqnarray}
&&\left|W_f(a,b)-\sum_{k=1}^K\int_\RR A_k(b)s_k(2\pi\phi_k(t))\frac{1}{\sqrt{a}}\overline{\psi\left(\frac{t-b}{a}\right)}\ud t\right|\label{Clemma1:eq3}\\
&\leq& \sum_{k=1}^K \int_\RR |A_k(b)-A_k(t)||s_k(2\pi\phi_k(t))|\frac{1}{\sqrt{a}}\left|\overline{\psi\left(\frac{t-b}{a}\right)}\right|\ud t\nonumber\\
&\leq&\epsilon\sum_{k=1}^K \|s_k\|_\infty\int_\RR |t-b|\left(|\phi'_k(b)|+\frac{1}{2}M_k''|t-b|\right)\frac{1}{\sqrt{a}}\left|\overline{\psi\left(\frac{t-b}{a}\right)}\right|\ud t\nonumber\\
&\leq&\epsilon\sum_{k=1}^K \|s_k\|_\infty \left(\phi_k'(b)a^{3/2}I_1+\frac{1}{2}M''_k a^{5/2}I_2\right)\nonumber
\end{eqnarray}
where the second inequality comes from Lemma \ref{est00}.

When $l\in\{1,\ldots,K\}$, $n\in\{1,\ldots,D\}$ and $(a,b)\in\RR_+\times\RR$ such that $|an\phi'_l(b)-1|\leq \Delta$, (\ref{Clemma1:eq1}) and (\ref{Clemma1:eq3}) together lead to the result:
\begin{eqnarray}
& &|W_f(a,b)-A_l(b)\hat{s}_l(n) e^{i2\pi n\phi_l(b)}\sqrt{a}\overline{\hat{\psi}(an\phi_l'(b))}|\label{Clemma1:eq4}\\
&\leq& \left|\sum_{k=1,k\neq l}^KA_k(b)\hat{s}_k(n) e^{i2\pi n\phi_k(b)}\sqrt{a}\overline{\hat{\psi}(an\phi_k'(b))}+\sum_{k=1}^K\sum_{j\in\NN,j\neq n}A_k(b)\hat{s}_k(j) e^{i2\pi j\phi_k(b)}\sqrt{a}\overline{\hat{\psi}(a j\phi_k'(b))}\right|\nonumber\\
&&+\epsilon\sum_{k=1}^K \|s_k\|_\infty \left(\phi_k'(b)a^{3/2}I_1+\frac{1}{2}M''_k a^{5/2}I_2\right)\nonumber\\
&&+2\pi\epsilon\sum_{k=1}^K A_k(b) \left[\sum_{n\in\ZZ} |n|\left|\hat{s}_k(n)\right|\right] \left(\frac{1}{2}a^{5/2}I_2|\phi'_k(b)|+\frac{1}{6}M_k''a^{7/2}I_3\right)\nonumber
\end{eqnarray}

By the definition of IMTs, the second term and the third term in (\ref{Clemma1:eq4}) are bounded by $$\epsilon \left\{\sum_{k=1}^K\|s_k\|_\infty \left(\phi_k'(b)a^{3/2}I_1+\frac{M''_k}{2} a^{5/2}I_2\right)+\pi A_k(b) \left[\sum_{n\in\ZZ} |n|\left|\hat{s}_k(n)\right|\right] \left(a^{5/2}I_2|\phi'_k(b)|+\frac{M''_k}{3}a^{7/2}I_3\right)\right\}.
$$ 
Moreover, when (\ref{thm:special1}) holds, due to Lemma \ref{est0} and the assumption of $\psi$, the first term in (\ref{Clemma1:eq4}) is bounded by
\begin{eqnarray}
&&\left|\sum_{k=1,k\neq l}^KA_k(b)\hat{s}_k(n) e^{i2\pi n\phi_k(b)}\sqrt{a}\overline{\hat{\psi}(an\phi_k'(b))}+\sum_{k=1}^K\sum_{j\in\NN,j\neq n}A_k(b)\hat{s}_k(j) e^{i2\pi j\phi_k(b)}\sqrt{a}\overline{\hat{\psi}(a j\phi_k'(b))}\right|\nonumber\\
&=&\left|\sum_{k=1}^K\sum_{j\in\NN,j\neq n}A_k(b)\hat{s}_k(j) e^{i2\pi j\phi_k(b)}\sqrt{a}\overline{\hat{\psi}(a j\phi_k'(b))}\right|\nonumber\\
& \leq & \sqrt{a} I_0\sum_{k=1}^{l-1} A_k(b) \sum_{j\in\NN,j>D}|\hat{s}_k(j)|\leq \epsilon\sqrt{a} I_0\sum_{k=1}^{l-1} A_k(b),
\end{eqnarray}
where the first equality comes from the fact that $(a,b)\in Z_{l,n}$ and the assumption (\ref{thm:special1}), the first inequality holds since $|\hat{\psi}(a j\phi_k'(b))|\leq I_0$ and Lemma \ref{est0} and the second inequality holds since $f\in \mathcal{C}^{\delta,D}_{\epsilon,d}$. To be more precise in the first inequality, when $k=l$, $\hat{\psi}(aj\phi_l'(b))\neq 0$ holds only for $j=n$;  when $k>l$, no $j\in\NN$ leads to $\hat{\psi}(aj\phi'_k(b))\neq 0$; when $k<l$, $\hat{\psi}(aj\phi'_k(b))\neq 0$ is possible only when $j>D$. 
The proof is thus done.
\end{proof}

In the next Lemma, we show that by differentiating the continuous wavelet transform, the instantaneous frequency $\phi_k'(b)$ and its multiples pop out. Though the instantaneous frequency is mixed up with other quantities, we will show that, when handled properly, this gives us an estimate of the instantaneous frequency with high accuracy. This Lemma is analogous to Lemma 3.9 in \cite{Daubechies2010}.
\begin{lemma}\label{est2}

For $l\in\{1,\ldots,K\}$, $n\in\{1,\ldots,D\}$ and $(a,b)\in Z_{l,n}$, we have
\begin{equation}\label{est2:rslt1}
\left|-i\partial_bW_f(a,b)-2\pi A_l(b)\hat{s}_l(n)e^{i2\pi n\phi_l(b)}n\phi'_l(b)\sqrt{a}\overline{\widehat{\psi}\left(an\phi'_l(b)\right)}\right|\leq
\epsilon a^{1/2}\Lambda_2(a,b),
\end{equation}
where
\begin{eqnarray}
\Lambda_2(a,b)&=&\sum_{k=1}^K\Big\{\|s_k\|_\infty \left(\phi_k'(b)I'_1+\frac{1}{2}M''_k aI'_2\right)+\pi A_k(b) \left[\sum_{j\in\ZZ} |j|\left|\hat{s}_k(j)\right|\right] \left(aI'_2|\phi'_k(b)|+\frac{1}{3}M_k''a^{2}I'_3\right)\Big\}\nonumber\\
&&+2\pi I_0\sum_{k=1}^{l-1} A_k(b)\phi'_k(b)\nonumber
\end{eqnarray}
and $I'_i=\int_\RR |x|^i|\psi'(x)|\ud x$. 
\end{lemma}

\begin{proof}
The proof follows the same lines as that for Lemma \ref{est1}. Under the same conditions, we can evaluate the following approximations. First,
\begin{eqnarray}
& &\int_\RR A_k(b)s_k(2\pi(\phi_k(b)-b\phi_k'(b))+2\pi\phi_k'(b)t)\frac{1}{a^{3/2}}\overline{\psi'\left(\frac{t-b}{a}\right)}\ud t\label{Clemma2:eq1}\\
&=& i2\pi A_k(b)\sum_{n\in\NN}\hat{s}_k(n)n\phi'_k(b) e^{i2\pi n\phi_k(b)}\sqrt{a}\overline{\hat{\psi}(an\phi_k'(b))}\nonumber.
\end{eqnarray}
Second,
\begin{eqnarray}
& &\left|\int_\RR A_k(b)s_k(2\pi\phi_k(t))\frac{1}{a^{3/2}}\overline{\psi'\left(\frac{t-b}{a}\right)}\ud t\right.\label{Clemma2:eq2}\\
&&\qquad \left.-\int_\RR A_k(b)s_k(2\pi(\phi_k(b)-b\phi_k'(b))+2\pi\phi_k'(b)t)\frac{1}{a^{3/2}}\overline{\psi'\left(\frac{t-b}{a}\right)}\ud t\right|\nonumber\\
&\leq&2\pi A_k(b)\int_\RR \left[\sum_{n\in\ZZ} |n|\left|\hat{s}_k(n)\right|\left| \int^{t-b}_0\left[\phi_k'(b+u)-\phi_k'(b)\right]\ud u\right|\right]\frac{1}{a^{3/2}}\left|\overline{\psi'\left(\frac{t-b}{a}\right)}\right|\ud t\nonumber\\
&\leq&2\pi \epsilon A_k(b) \left[\sum_{n\in\ZZ} |n|\left|\hat{s}_k(n)\right|\right] \left(\frac{1}{2}a^{3/2}I'_2|\phi'_k(b)|+\frac{1}{6}M_k''a^{5/2}I'_3\right),\nonumber
\end{eqnarray}
where $\sum_{n\in\ZZ} |n|\left|\hat{s}_k(n)\right|<\infty$ since $s_k\in C^{1,\alpha}$ and $\alpha>1/2$. And third,
\begin{eqnarray}
&&\left|-\partial_bW_f(a,b)-\sum_{k=1}^K\int_\RR A_k(b)s_k(2\pi\phi_k(t))\frac{1}{a^{3/2}}\overline{\psi'\left(\frac{t-b}{a}\right)}\ud t\right|\label{Clemma2:eq3}\\
&\leq&\epsilon\sum_{k=1}^K \|s_k\|_\infty \left(\phi_k'(b)a^{1/2}I'_1+\frac{1}{2}M''_k a^{3/2}I'_2\right).\nonumber
\end{eqnarray}
When $l\in\{1,\ldots,K\}$, $n\in\{1,\ldots,D\}$ and $(a,b)\in\RR_+\times\RR$ such that $|an\phi'_l(b)-1|\leq \Delta$, (\ref{Clemma2:eq1})-(\ref{Clemma2:eq3}) together lead to the result:
\begin{eqnarray}
& &|-i\partial_bW_f(a,b)-2\pi A_l(b)\hat{s}_l(n) e^{i2\pi n\phi_l(b)}n\phi_l'(b)\sqrt{a}\overline{\hat{\psi}(an\phi_l'(b))}|\nonumber\\
&\leq& 2\pi\left|\sum_{k=1,k\neq l}^KA_k(b)\hat{s}_k(n)e^{i2\pi n\phi_k(b)}n\phi_k'(b) \sqrt{a}\overline{\hat{\psi}(an\phi_k'(b))}+\sum_{k=1}^K \sum_{j\in\NN,j\neq n}A_k(b)\hat{s}_k(j) e^{i2\pi j\phi_k(b)}j\phi'_k(b)\sqrt{a}\overline{\hat{\psi}(aj\phi_k'(b))}\right|\nonumber\\
&&+\epsilon\sum_{k=1}^K \|s_k\|_\infty \left(\phi_k'(b)a^{1/2}I'_1+\frac{1}{2}M''_k a^{3/2}I'_2\right)\nonumber\\
&&+2\pi \epsilon \sum_{k=1}^KA_k(b) \left[\sum_{j\in\ZZ} |j|\left|\hat{s}_k(j)\right|\right] \left(\frac{1}{2}a^{3/2}I'_2|\phi'_k(b)|+\frac{1}{6}M_k''a^{5/2}I'_3\right).\nonumber
\end{eqnarray}
By the same argument as that for Lemma \ref{est1}, we get (\ref{est2:rslt1}).
\end{proof}
}

The same remark for Lemma \ref{est1} holds for Lemma \ref{est2}. That is, the larger the difference between the wave shape functions and the cosine function, the larger the possible negative effect on the precision of the estimate. The following Lemma clarifies the role of the function $\omega_f(a,b)$. Indeed, it states that $\omega_f(a,b)$ provides the information of the instantaneous frequency of each component. Since the proof of this Lemma is the same as that of Estimate 3.8 in \cite{Daubechies2010}, we shall skip it.
\begin{lemma}\label{est3}
For $l \in\{1,\ldots,K\}$, $n\in\{1,\ldots,D\}$ and $(a,b)\in Z_{l,n}$ so that $|W_f(a,b)| \geq \weps$, we have
\[
\left|\omega_f(a,b)-n\phi'_l(b)\right|\leq \epsilon^{2/3}a^{1/2}(\Lambda_1n\phi'_l(b)+\Lambda_2/2\pi).
\]

\end{lemma}

It is now clear that with appropriate restrictions on $\epsilon$, the first two claims in Theorem \ref{mainthm1} are proved. Indeed, note that by the definition of $Z_{k,n}$ and the uniform lower and upper bounds on $\phi'_k(b)$, we know that the values of $a$ for which $(a,b) \in \cup_{n=1}^D\cup_{k=1}^K Z_{k,n}$ are uniformly bounded. Then, supposing that (\ref{thm:special1}) holds, due to the uniform boundedness of $A_k(b)$ and $\phi'_k(b)$, it follows that $\Lambda_1(a,b)$ and $\Lambda_2(a,b)$ are uniformly bounded in $\cup_{n=1}^D\cup_{k=1}^K Z_{k,n}$ as well. Thus, for $(a,b) \in \cup_{n=1}^D\cup_{k=1}^K Z_{k,n}$, there exists $\epsilon>0$ so that
\begin{equation}\label{condeps1}
\epsilon<a^{-3/4}\,\Lambda_1^{-3/2}~,
\end{equation}
which leads to $\epsilon a^{1/2}\Lambda_1(a,b)<\weps $. If further we impose the condition that for all $l=1,\ldots,K$ and $n=1,\ldots,D$,
\begin{equation}\label{condeps2}
\epsilon< a^{-3/2}(\Lambda_1n\phi'_l(b)+\Lambda_2/2\pi)^{-3},
\end{equation}
then $\epsilon^{2/3} a^{1/2}(\Lambda_1n\phi'_l(b)+\Lambda_2/2\pi)<\weps $. Thus the first two claims in Theorem \ref{mainthm1} hold.

The final Lemma concerns the reconstruction of each component; when $s_k(t)=e^{i t}$ for all $k=1,...,K$, we recover Lemma 3.9 in \cite{Daubechies2010}. 
\begin{lemma}\label{est4}
Suppose that both (\ref{condeps1}) and (\ref{condeps2}) are satisfied, and that, in addition, for all $b$ and $k\in\{1,\ldots,K\}$ under consideration,
\begin{equation}\label{condeps3}
\epsilon \leq \min\left\{\frac{d^3 [\phi'_1(b)+\phi'_2(b)]^3}{8}, \frac{27d^3 \phi'_1(b)^3}{8}\right\}.
\end{equation}
Then for any $b \in \RR$
\[
\left|\lim_{\alpha\rightarrow 0}\left(\mathcal{R}_{\psi}^{-1}\int_{\cup_{n=1}^D\{\xi:\,|\xi-n\phi'_k(b)|<\weps\}}S^{\alpha,\gamma}_{f}(b,\xi)\ud\xi\right) -A_k(b)s_k(2\pi \phi_k(b))\right|\leq C\weps,
\]
where $C=A_k(b)\weps^2+4\left[\left(\frac{\phi'_k(b)}{1-\Delta}\right)^{1/2}-\left(\frac{\phi'_k(b)}{1+\Delta}\right)^{1/2}\right]< A_k(b)+4\left[\left(\frac{\phi'_k(b)}{1-\Delta}\right)^{1/2}-\left(\frac{\phi'_k(b)}{1+\Delta}\right)^{1/2}\right]$.
\end{lemma}

\begin{proof}
Fix $b\in\RR$. As a function a $a$, $W_f(a,b)\in C^\infty(A_{\weps,f}(b))$, so by the definition of the wavelet Synchrosqueezing transform, as a function of $\xi$, $S^{\alpha,\weps}_{f}(b,\xi)\in C^\infty(\RR)$. Thus, we have
\begin{eqnarray}
&&\lim_{\alpha\rightarrow 0}\int_{\cup_{n=1}^D\{\xi:\,|\xi-n\phi'_k(b)|<\weps\}}S^{\alpha,\weps}_{f}(b,\xi)\ud\xi\label{lemma4:est1}\\
&=& \lim_{\alpha\to 0}\int_{\cup_{n=1}^D\{\xi:\,|\xi-n\phi'_k(b)|<\weps\}}\int_{A_{\weps,f}(b)} W_f(a,b)\frac{1}{\alpha}h\left(\frac{|\xi-\omega_f(a,b)|}{\alpha}\right)a^{-3/2} \ud a\ud\xi\nonumber \\
&=&\lim_{\alpha\to 0}\int_{A_{\weps,f}(b)}a^{-3/2}W_f(a,b) \int_{\cup_{n=1}^D\{\xi:\,|\xi-n\phi'_k(b)|<\weps\}}\frac{1}{\alpha}h \left(\frac{|\xi-\omega_f(a,b)|}{\alpha}\right)\ud\xi \ud a\nonumber\\
&=&\int_{A_{\weps,f}(b)}\lim_{\alpha \to 0}a^{-3/2}W_f(a,b)\int_{\cup_{n=1}^D\{\xi:\,|\xi-n\phi'_k(b)|<\weps\}}\frac{1}{\alpha}h\left(\frac{|\xi-\omega_f(a,b)|}{\alpha}\right)\ud\xi \ud a\nonumber\\
&=&\int_{ A_{\weps,f}(b)\cap \left(\cup_{n=1}^D \{a: |\omega_f(a,b)-n\phi'_k(b)|<\weps \}\right)} W_f(a,b) a^{-3/2} \ud a\nonumber\, ,
\end{eqnarray}
where we have used Fubini's theorem for the second equality, the Dominant Convergence theorem
for the third equality, and the approximation of identity for the fourth equality. Indeed, the integrand on the third line is bounded by $a^{-3/2}|W_f(a,b)|\in L^1(A_{\weps,f}(b))$ 
and converges almost everywhere to $a^{-3/2}W_f(a,b)$ if $|\omega_f(a,b)-n\phi'_k(b)|<\weps$ for some $n\in\{1,\ldots,D\}$, and to zero otherwise. 

We now claim that there is only one $l\in\{1,\ldots,K\}$ and $n\in\{1,\ldots,D\}$ for which $|an\phi_l'(b) - 1 | <\Delta$ and $|\omega_f(a,b)-n\phi'_l(b)|<\weps$ hold simultaneously. In fact, if there exist $l'\in\{1,\ldots,K\}$ and $n'\in\{1,\ldots,D\}$ so that $|an'\phi'_{l'}(b)-1|<\Delta$ holds, where $l\neq l'$ or $n\neq n'$, we get
\begin{equation*}
|\omega_f(a,b)-n'\phi'_{l'}(b)| \geq |n\phi'_{l}(b)-n'\phi'_{l'}(b)|-|\omega_f(a,b)-n\phi'_l(b)|\geq d[n\phi'_{l}(b)+n'\phi'_{l'}(b)]-\weps,
\end{equation*}
where the second inequality comes from (\ref{est0:eq2}) and Lemma \ref{est3}. Notice that
\begin{equation}\label{ax}
d [n\phi'_{l}(b)+n'\phi'_{l'}(b)] \geq d\min\left\{ [\phi'_1(b)+\phi'_2(b)], 3\phi'_1(b)\right\} \geq 2 \weps,
\end{equation}
where the first inequality holds since $n\phi'_l(b)\geq \phi'_1(b)$, $n'\phi'_{l'}(b)\geq \phi'_1(b)$ and $\phi'_2(b)$ might be larger than $2\phi'_1(b)$ and the second inequality comes from (\ref{condeps3}). Thus we conclude that 
\begin{equation*}
|\omega_f(a,b)-n'\phi'_{l'}(b)| >\weps,
\end{equation*}
which is absurd.
Next, from Lemma \ref{est2} and (\ref{condeps1}) we
know that $|W_f(a,b)|> \weps$ only when $|an\phi_l'(b) - 1 | < \Delta$ for some $l \in \{1,\ldots,K\}$ and $n\in\{1,\ldots,D\}$. Hence we know 
$$
A_{\weps,f}(b)\cap \left(\cup_{n=1}^D \{a: |\omega_f(a,b)-n\phi'_k(b)|<\weps \}\right)=A_{\weps,f}(b)\cap \left(\cup_{n=1}^D \{a:|an\phi_k'(b) - 1|<\Delta\}\right)
$$ 
and the right hand side of (\ref{lemma4:est1}) becomes
\begin{eqnarray}
&&\int_{ A_{\weps,f}(b)\cap \left(\cup_{n=1}^D \{a:|an\phi_k'(b) - 1|<\Delta\}\right)} W_f(a,b) a^{-3/2} \ud a\nonumber\\
&=&\int_{\cup_{n=1}^D\{a:|an\phi_k'(b) - 1|<\Delta\}}W_f(a,b) a^{-3/2}\ud a - \int_{ \left(\cup_{n=1}^D\{|an\phi_k'(b) -1|<\Delta\}\right)\backslash A_{\weps,f}(b)}W_f(a,b)a^{-3/2}\ud a~.\nonumber
\end{eqnarray}
Thus we obtain
\begin{eqnarray}
&&\left|\lim_{\alpha\to 0}\mathcal{R}_{\psi}^{-1}\int_{\cup_{n=1}^D\{\xi:\,|\xi-n\phi'_k(b)|<\weps\}}S^{\alpha,\weps}_{f}(b,\xi)\ud\xi -A_k(b)s_k(2\pi\phi_k(b))\right|\nonumber\\
&\leq& \left|\mathcal{R}_{\psi}^{-1}\left(\int_{\cup_{n=1}^D\{a:|an\phi_k'(b) - 1|<\Delta\}}W_f(a,b) a^{-3/2}\ud a \right)-A_k(b)s_k(2\pi \phi_k(b)) \right| \nonumber\\
&&+\mathcal{R}_{\psi}^{-1}\left|\int_{ \left(\cup_{n=1}^D\{|an\phi_k'(b) -1|<\Delta\}\right)\backslash A_{\weps,f}(b)}W_f(a,b)a^{-3/2}\ud a \right| \nonumber\\
&\leq&\left|\mathcal{R}_{\psi}^{-1}A_k(b)\sum_{n=1}^D\hat{s}_k(n)e^{i2\pi n\phi_k(b)}\left(\int_{|an\phi'_k(b)-1|<\Delta}\sqrt{a}\overline{\widehat{\psi}(an \phi'_k(b))}a^{-3/2}\ud a\right)-A_k(b)s_k(2\pi\phi_k(b))\right|\nonumber\\
&&+\mathcal{R}_{\psi}^{-1}\sum_{n=1}^D\int_{|an\phi'_k(b)-1|<\Delta}\weps a^{-3/2}\ud a+\mathcal{R}_{\psi}^{-1}\left|\int_{\left(\cup_{n=1}^D\{|an\phi_k'(b) -1|<\Delta\}\right)\backslash A_{\weps,f}(b)}W_f(a,b)a^{-3/2} \ud a \right|\nonumber\\
&\leq&\left|A_k(b)\sum_{n=1}^D\hat{s}_k(n)e^{i2\pi n\phi_k(b)}-A_k(b)s_k(2\pi\phi_k(b))\right|+2\mathcal{R}_{\psi}^{-1}\sum_{n=1}^D\int_{|an\phi'_k(b)-1|<\Delta}\weps a^{-3/2}\ud a\label{lemma4:est3}
\end{eqnarray}
where the second inequality comes from Lemma \ref{est1} and \ref{condeps1} and the third inequality comes from Lemma \ref{est1} and the fact that
\[
\int_{|a n\phi'_k(b)-1|<\Delta}\overline{\widehat{\psi}(an\phi'_k(b))} a^{-1}\ud a = \int_{|\zeta-1|<\Delta}\overline{\widehat{\psi}(\zeta)}\zeta^{-1}\ud\zeta=\mathcal{R}_\psi.
\]
The first term in (\ref{lemma4:est3}) is bounded by
\[
\left| A_k(b)\sum_{n=D+1}^\infty\hat{s}_k(n)e^{i2\pi n\phi_k(b)} \right| \leq A_k(b)\sum_{n=D+1}^\infty|\hat{s}_k(n)|\leq \epsilon A_k(b).
\]
The second term in (\ref{lemma4:est3}) can be worked out explicitly:
$$
2\mathcal{R}_{\psi}^{-1}\sum_{n=1}^D\int_{|an\phi'_k(b)-1|<\Delta}\weps a^{-3/2}\ud a=4\weps\sum_{n=1}^D\left[\left(\frac{n\phi'_k(b)}{1-\Delta}\right)^{1/2}-\left(\frac{n\phi'_k(b)}{1+\Delta}\right)^{1/2}\right].
$$
We thus conclude
\begin{eqnarray}
&&\left|\lim_{\alpha\to 0}\mathcal{R}_{\psi}^{-1}\int_{\cup_{n=1}^D\{\xi:\,|\xi-n\phi'_k(b)|<\weps\}}S^{\alpha,\weps}_{f}(b,\xi)\ud\xi -A_k(b)s_k(2\pi\phi_k(b))\right|\nonumber\\
&\leq&
\epsilon A_k(b)+4\weps\left[\left(\frac{\phi'_k(b)}{1-\Delta}\right)^{1/2}-\left(\frac{\phi'_k(b)}{1+\Delta}\right)^{1/2}\right]\nonumber.
\end{eqnarray}
\end{proof}

Combining Lemmas \ref{est1}-\ref{est4} completes the proof of Theorem \ref{mainthm1}.
\begin{remark}
This Theorem suggests that if we know a priori that the components of the signal have wave shape function close to an imaginary exponential (or if its real part has components close to a cosine function), then focusing on the band around the instantaneous frequency $\phi'_k(b)$ and its multiples allows us to recover the signal. \end{remark}

Now we demonstrate some numerical results of applying the Synchrosqueezing transform to analyze the functions in $\mathcal{C}^{\delta,D}_{\epsilon,d}$. 
Take the phase functions $\phi_1(t)=1.5t+0.2\cos(t+1)$ and $\phi_2(t)=4.5(t+0.2\cos(t))$ and the amplitude modulation functions $A_1(t)=1+0.1\sin(t^{1.1})$ and $A_2(t)=\sqrt{1+\cos(t)}$. Consider the following two functions for comparison:
$$
f_1(t)=A_1(t)s_{\II}(2\pi\phi_1(t))+A_2(t)s_1(2\pi\phi_2(t))
$$
and
$$
f_2(t)=\frac{A_1(t)}{3.5}\cos(2\pi\phi_1(t))+A_2(t)s_1(2\pi\phi_2(t)),
$$
where $s_1(t)$ and $s_{\II}(t)$ are the shape functions demonstrated in Figure \ref{shape3}.
Clearly, $f_1$ and $f_2$ are both composed of one component with instantaneous frequency $1.5-0.2\sin(t+1)$ and one component with instantaneous frequency $4.5-0.9\sin(t)$ but with different wave shape functions in the low frequency component.

As is shown in Figure \ref{shape3}, the coefficients of the high Fourier modes of $s_{\II}$ are quite significant, whereas those of $\frac{1}{3.5}\cos(t)$ are all 0. In this case, since $\phi_2$ is roughly $3$ times $\phi_1$, it is the coefficient of the third Fourier mode of $s_{\II}$ that matters. Thus, according to Theorem \ref{mainthm1}, we expect to have worse $\phi'_2$ estimation from $f_1$, which is borne out by Figure \ref{ex1}.

\begin{figure}[h]
\subfigure{
\includegraphics[width=0.95\textwidth]{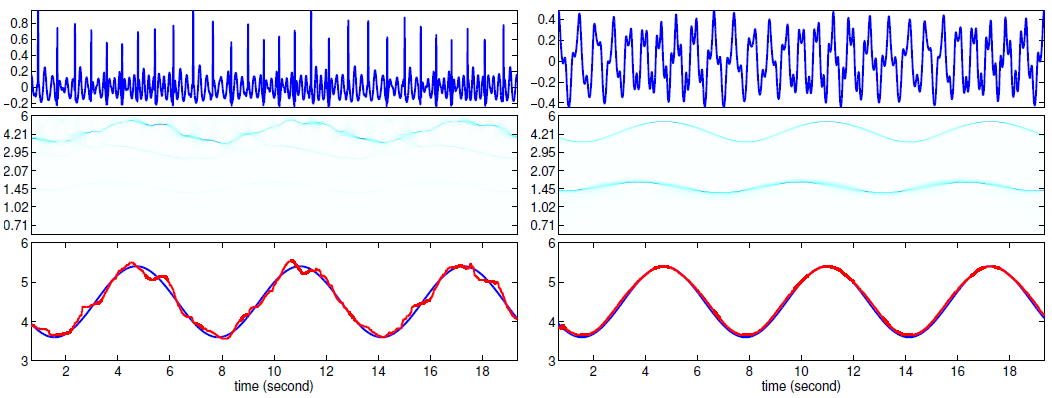}}
\caption{Left top: the $f_1$ signal; left middle: the Synchrosqueezing transform of $f_1$, where the y-axis is demonstrated in the log scale; left bottom: the blue curve is $\phi'_2(t)$, while the red curve is the estimated $\phi'_2$; right top: the $f_2$ signal; right middle: the Synchrosqueezing transform of $f_2$, where the y-axis is demonstrated in the log scale; right bottom: the blue curve is $\phi'_2(t)$, while the red curve is the estimated $\phi'_2$. It is clear that the estimation of $\phi'_2$ is worse in $f_1$. Notice that the dominant curve around $\phi'_1$ in the left middle figure is not very marked. This is caused by the small coefficient of the first Fourier mode of $\hat{s}_{\II}$.}\label{ex1}
\end{figure}

Next we show the component reconstruction results stated in Theorem \ref{mainthm1}. Take the phase functions $\phi_3(t)=4t+0.1\cos(t+1)$ and $\phi_4(t)=5t+0.2\cos(t)^2$ and the wave shape functions $s_3(t)$ and $s_4(t)$. The Fourier modes of $s_3(t)$ and $s_4(t)$ are illustrated in Figure \ref{shape_ex2}. Consider the function
\[
f_3(t)=f_{3,1}(t)+f_{3,2}(t),
\]
where 
\[
f_{3,1}(t)=A_1(t)s_3(2\pi\phi_3(t))
\]
and
\[
f_{3,2}(t)=A_2(t)s_4(2\pi\phi_4(t)).
\]
Since $f_{3,1}(t)\in C^{1.2, 3}_{\epsilon}$ and $f_{3,2}(t)\in C^{0.5,4}_{\epsilon}$, we know $f_3(t)\in C^{1.2,4}_{\epsilon,d}$, where $\epsilon\approx 2/25$ and $d\approx1/9$. We take $\Delta=0.1$ in the simulation, and it is clear that the intervals $\left[\frac{1-\Delta}{n\phi'_3(t)},\frac{1+\Delta}{n\phi'_3(t)}\right]$ for $n=1,\ldots,4$ and $\left[\frac{1-\Delta}{m\phi'_4(t)},\frac{1+\Delta}{m\phi'_4(t)}\right]$ for $m=1,\ldots,3$ do not overlap, so the condition (\ref{thm:special1}) is satisfied, and hence we expect to be able to reconstruct $f_{3,1}(t)$ and $f_{3,2}(t)$ from $f_3$ via Synchrosqueezing transform. The extraction result is shown in the left panel of Figure \ref{ex2}. We also demonstrate the robustness of Synchrosqueezing transform in this case. Consider 
$$
f_4=f_3+\sigma W,
$$ 
where $W$ is the white noise with variance $1$ and $\sigma=\sqrt{\text{var}(f_3)10^{1/5}}$, that is, we add $-2$dB white noise to $f_3$ if the signal-to-noise ratio is defined by 
$$
\text{signal to noise ratio }(\text{dB})=10\log_{10}\left(\frac{\text{var}(f_3)}{\sigma^2}\right).
$$
We first demonstrate the Synchrosqueezing transform of $f_4$ and the reassigned Morlet scalogram \cite{auger_flandrin:1995} in Figure \ref{ex2:tfa}. Then we extract the components $f_{3,1}$ and $f_{3,2}$ by applying Synchrosqueezing transform on $f_4$. The result is shown in the right panel of Figure \ref{ex2}. Since reconstructing the components from the reassigned Morlet scalogram is not guaranteed, we only demonstrate the reconstruction results of the Synchrosqueezing transform. 

\begin{figure}[h]
\subfigure{
\includegraphics[width=0.95\textwidth]{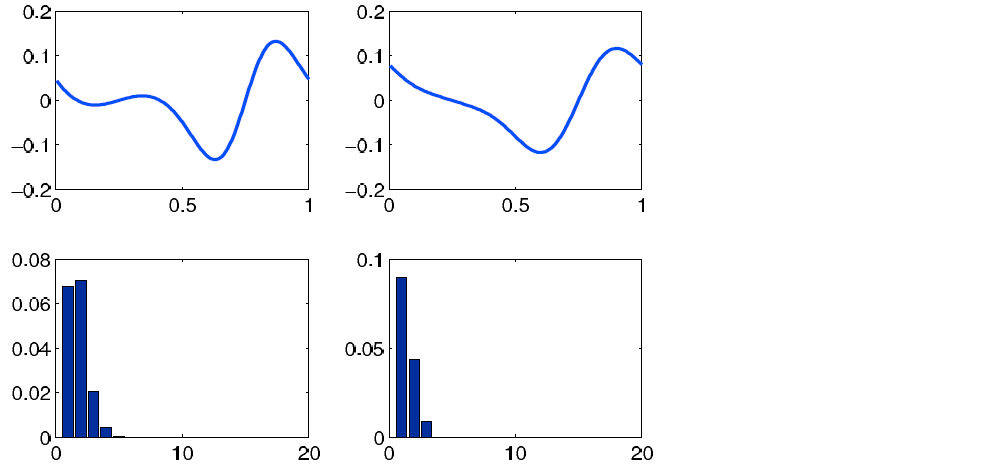}}
\caption{Top row: the $s_3$ (left) and $s_4$ (right) wave shape function; bottom row: the first $20$ Fourier modes of $s_3$ (left) and $s_4$ (right). Note that $\hat{s}_3(n)$ and $\hat{s}_4(n)$ are not zero but very small when $n>4$.}\label{shape_ex2}
\end{figure}

\begin{figure}[h]
\subfigure{
\includegraphics[width=0.95\textwidth]{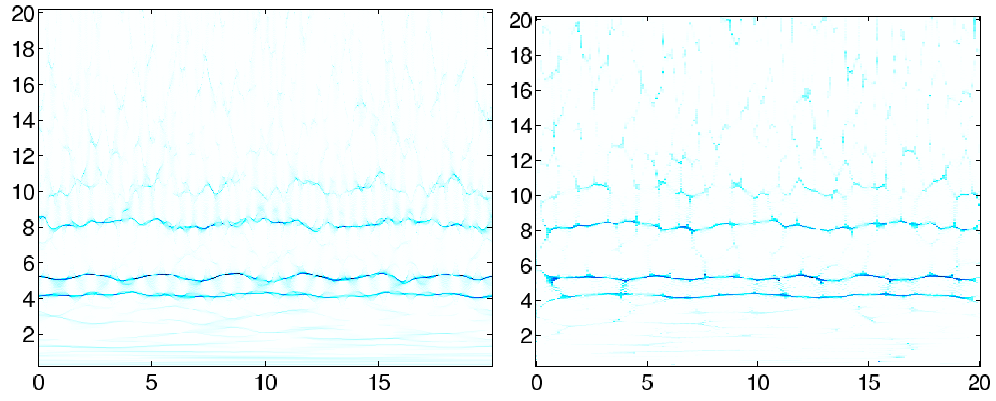}}
\caption{Left: the Synchrosqueezing transform of $f_4$; right: the reassigned Morlet scalogram. Notice that the performance of the Synchrosqueezing transform is equivalent to the reassigned Morlet scalogram for the purpose of extracting the instantaneous frequency.}\label{ex2:tfa}
\end{figure}

\begin{figure}[h]
\subfigure{
\includegraphics[width=0.95\textwidth]{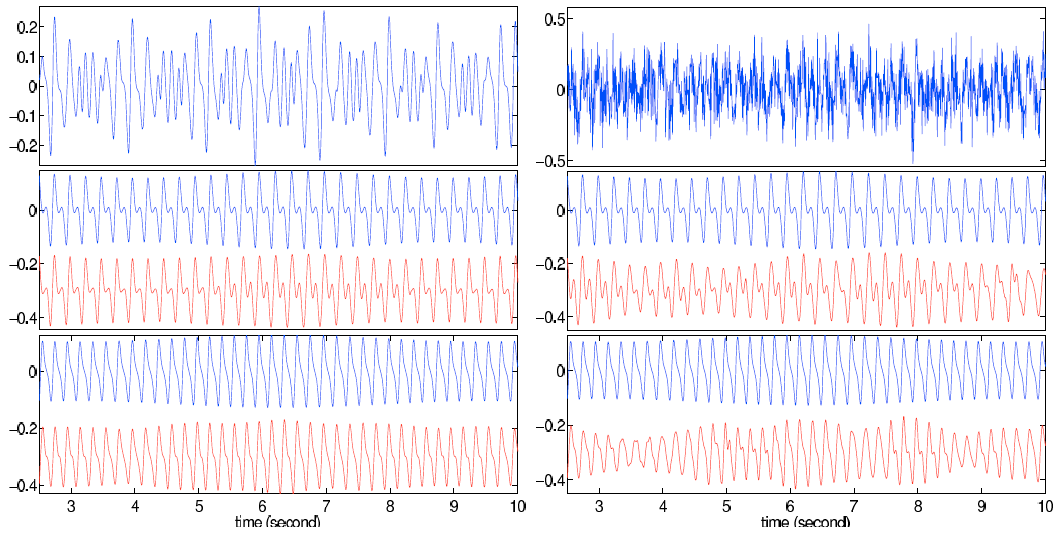}}
\caption{Left top: the $f_3$ signal; left middle: the blue curve is $A_1(t)s_3(2\pi\phi_3(t))$, while the red curve is the reconstructed $A_1(t)s_3(2\pi\phi_3(t))$, which is shifted down by $0.3$ to improve the readability; left bottom: the blue curve is $A_2(t)s_4(2\pi\phi_4(t))$, while the red curve is the reconstructed $A_2(t)s_4(2\pi\phi_4(t))$, which is shifted down by $0.3$ to improve the readability; right top: the $f_4$ signal; right middle: the blue curve is $A_1(t)s_3(2\pi\phi_3(t))$, while the red curve is the reconstructed $A_1(t)s_3(2\pi\phi_3(t))$, which is shifted down by $0.3$ to improve the readability; right bottom: the blue curve is $A_2(t)s_4(2\pi\phi_4(t))$, while the red curve is the reconstructed $A_2(t)s_4(2\pi\phi_4(t))$, which is shifted down by $0.3$ to improve the readability.}\label{ex2}
\end{figure}

\clearpage

\section{Acknowledgements}
The author acknowledges support by FHWA grant DTFH61-08-C-00028 and Award Number FA9550-09-1-0551 from AFOSR; he is also grateful for
 valuable discussions with Eugene Brevdo, Gaurav Thakur and Professor Ingrid Daubechies, and to Professor Chung-Kang Peng for kindly providing the respiratory signal during sleep and for many discussions. He thanks the anonymous reviewers for their useful recommendations to improve this paper.

\bibliographystyle{amsplain} \bibliographystyle{amsplain} \bibliographystyle{amsplain}
\bibliography{interIF}

\providecommand{\bysame}{\leavevmode\hbox to3em{\hrulefill}\thinspace}
\providecommand{\MR}{\relax\ifhmode\unskip\space\fi MR }
\providecommand{\MRhref}[2]{%
  \href{http://www.ams.org/mathscinet-getitem?mr=#1}{#2}
}
\providecommand{\href}[2]{#2}
\begin{thebibliography}{10}

\bibitem{abboud1987}
S.~Abboud, R.J. Cohen, A.~Selwyn, P.~Ganz, D.~Sadeh, and P.L Friedman,
  \emph{Detection of transient myocardial ischemia by computer analysis of
  standard and signal-averaged high-frequency electrocardiograms in patients
  undergoing percutaneous transluminal coronary angioplasty}, Circulation
  \textbf{76} (1987), 585--596.

\bibitem{auger_flandrin:1995}
F.~Auger and P.~Flandrin, \emph{Improving the readability of time-frequency and
  time-scale representations by the reassignment method}, Signal Processing,
  IEEE Transactions on \textbf{43} (1995), no.~5, 1068 --1089.

\bibitem{Benchetrit2000}
G.~Benchetrit, \emph{Breathing pattern in humans: diversity and individuality},
  Respiration Physiology \textbf{122} (2000), no.~2-3, 123 -- 129.

\bibitem{kou:2011}
M.-Y. Bien, Y.-S. Lin, C.-H. Shih, Y.-L. Yang, H.-W. Lin, K.-J. Bai, J.-H.
  Wang, and Y.~R. Kou, \emph{Comparisons of predictive performance of breathing
  pattern variability measured during t-piece, automatic tube compensation, and
  pressure support ventilation for weaning intensive care unit patients from
  mechanical ventilation}, Critical Care Medicine \textbf{39} (2011),
  2253--2262.

\bibitem{ID1992}
I.~Daubechies, \emph{Ten lectures on wavelets}, SIAM, 1992.

\bibitem{Daubechies2010}
I.~Daubechies, J.~Lu, and H.-T. Wu, \emph{Synchrosqueezed wavelet transforms:
  An empirical mode decomposition-like tool}, Applied and Computational
  Harmonic Analysis (2010).

\bibitem{Daubechies1996}
I.~Daubechies and S.~Maes, \emph{{A nonlinear squeezing of the continuous
  wavelet transform based on auditory nerve models}}, Wavelets in Medicine and
  Biology (1996), 527--546.

\bibitem{ecg}
A.~L. Goldberger, \emph{Clinical electrocardiography: A simplified approach},
  Mosby, 2006.

\bibitem{Guyton:00}
A.~Guyton and J.~Hall, \emph{Textbook of medical physiology}, Saunders, 2000.

\bibitem{HuangWuLong:09}
N.~E. Huang, Z.~Wu, S.~R. Long, K.~C. Arnold, K.~Blank, and T.~W. Liu, \emph{On
  instantaneous frequency}, Advances in Adaptive Data Analysis \textbf{1}
  (2009), 177--229.

\bibitem{keener1998}
J.~Keener, \emph{Mathematical physiology}, Springer, 1998.

\bibitem{kodera_gendrin_villedary:1978}
K.~Kodera, R.~Gendrin, and C.~Villedary, \emph{Analysis of time-varying signals
  with small bt values}, Acoustics, Speech and Signal Processing, IEEE
  Transactions on \textbf{26} (1978), no.~1, 64 -- 76.

\bibitem{hrv}
M.~Malik and A.~J. Camm, \emph{Heart rate variability}, Wiley-Blackwell, 1995.

\bibitem{mor1990}
V.~Mor-Avi and S.~Akselrod, \emph{Spectral analysis of canine
  epicardialelectrogram. short-term variations in the frequency content induced
  by myocardial ischemia}, Circulation Research \textbf{66} (1990), 1681--1691.

\bibitem{picinbono:1997}
B.~Picinbono, \emph{On instantaneous amplitude and phase of signals}, Signal
  Processing, IEEE Transactions on \textbf{45} (1997), no.~3, 552 --560.

\bibitem{picinbono_martin:1983}
B.~Picinbono and W.~Martin, \emph{Representation des signaux par amplitude et
  phase instantanes}, Annales des Tlcommunications \textbf{38} (1983),
  179--190.

\bibitem{RK}
A.~Rechtschaffen and A.~Kales, \emph{A manual of standardized terminology,
  techniques and scoring system for sleep stages of human subjects},
  Washington: Public Health Service, US Government Printing Office, 1968.

\bibitem{brevdo_fuckar_thakur_wu:2012}
G.~Thakur, E.~Brevdo, N.~S. Fu\v{c}kar, and H.-T. Wu, \emph{The
  synchrosqueezing algorithm for time-varying spectral analysis: robustness
  properties and new paleoclimate applications},  (2012), submitted.
  \url{http://arxiv.org/abs/1105.0010}.

\bibitem{Thakur2010}
G.~Thakur and H.-T. Wu, \emph{{Synchrosqueezing-based Recovery of Instantaneous
  Frequency from Nonuniform Samples}}, SIAM J. Math. Anal. \textbf{43} (2011),
  no.~43, 2078--2095.

\bibitem{wu_flandrin_daubechies:2011}
H.-T. Wu, P.~Flandrin, and I.~Daubechies, \emph{{One or Two Frequencies? The
  Synchrosqueezing Answers}}, Adv. Adapt. Data Anal. \textbf{3} (2011), no.~1,
  29--39.

\bibitem{Wysocki2006}
M.~Wysocki, C.~Cracco, A.~Teixeira, A.~Mercat, J.~Diehl, Y.~Lefort, J~Derenne,
  and T~Similowski, \emph{Reduced breathing variability as a predictor of
  unsuccessful patient separation from mechanical ventilation}, Critical Care
  Medicine \textbf{34} (2006), 2076--2083.

\end{thebibliography}

\end{document}